\documentclass[11pt,a4paper]{article}
\usepackage{etex}
\usepackage{inputenc}
\usepackage[english]{babel}
\usepackage{soul}
\usepackage{lmodern}					% harmonie des polices
\usepackage{textcomp}				% table de symboles
\usepackage{verbatim}			% pour faire des commentaires par block avec begin{comment}...end{comment}
\usepackage[toc,page]{appendix} 		% pour faire des annexes
\usepackage{cite}
\usepackage{amsmath,					% packages mathmatiques
			amsthm,
			amsfonts,
			amssymb, 
			mathtools}%,a4wide}
\usepackage{stmaryrd}
\usepackage{physics}
\usepackage{enumitem}
\usepackage{hyperref}
\hypersetup{
colorlinks=true,
linkcolor=purple,
filecolor=magenta,
urlcolor=cyan,
citecolor=blue
}

 \usepackage{tikz-cd} 
\usetikzlibrary{tikzmark}
\usetikzlibrary{arrows}
 \usepackage{verbatim}
\definecolor{DarkGreen}{HTML}{1cad22}
\usepackage[hmargin=2cm,vmargin=2cm]{geometry}
\usepackage{graphicx} % Required for inserting images
\usepackage{amsthm, amssymb, amsmath}
\usepackage{color}
\usepackage{xcolor}
\usepackage{tikz}
\usepackage{tikz-cd}
\usepackage{amsmath}
\usepackage{amssymb}
\usepackage{float}
\usepackage{caption}
\usepackage{mathtools}
\usepackage{hyperref}
\usepackage{mathrsfs}

\usepackage{graphicx}
\usepackage{epsfig}
\newcommand{\e}{e}

\numberwithin{equation}{section}

{\theoremstyle{definition}\newtheorem{definition}{Definition}[section]

\newtheorem{defititle}[definition]{\Title}

\newtheorem{remark}[definition]{Remark}
\newtheorem{example}[definition]{Example}

}
\newtheorem{prop}[definition]{Proposition}
\newtheorem{proposition-definition}[definition]{Proposition-Definition}
\newtheorem{lemma}[definition]{Lemma}
\newtheorem{thm}[definition]{Theorem}
\newtheorem{cor}[definition]{Corollary}

\newtheorem*{theorem}{Theorem}

\newtheoremstyle{named}{}{}{\itshape}{}{\bfseries}{.}{.5em}{\thmnote{#3's }#1}
\theoremstyle{named}

\usepackage{amsthm}

\newenvironment{mainthm-env}[1]
  {\begin{theorem}[Main Theorem #1]}
  {\end{theorem}}

\title{Obstructions to lifting quaternionic torus actions}

\author{Panagiotis Batakidis$^\dagger$ \\\href{mailto:batakidis@math.auth.gr}{batakidis@math.auth.gr}  \and Ioannis Gkeneralis$^\dagger$
\\\href{mailto:igkeneralis@math.auth.gr}{igkeneralis@math.auth.gr}}

\date{$^\dagger$ Department of Mathematics, Aristotle University of Thessaloniki}

\newcommand{\keywords}[1]{\textbf{Keywords: } #1}
\newcommand{\msc}[1]{\textbf{Mathematics Subject Classification: } #1}

\begin{document}

\maketitle

\begin{abstract}
We study the problem of lifting global and local quaternionic torus actions to principal quaternionic torus bundles. Let \(Q^k=(\operatorname{Sp}(1))^k\cong (S^3)^n\), let \(G\) be a compact Lie group acting on a connected, locally finite CW complex \(X\), and let \(Q^k\longrightarrow P\longrightarrow X\) be a principal \(Q^k\)-bundle. We first formulate a quaternionic analogue of the obstruction-theoretic framework of Hattori--Yoshida. The existence of a lifted \(G\)-action implies that the isomorphism class of \(P\) lies in the image of the restriction map induced by the Borel construction \(X_G=EG\times_GX\). In particular, the second Chern class admits an equivariant extension. Once a continuous pseudo-lift has been chosen, its failure to define a genuine action is measured by a factor set with values in the generally nonabelian gauge group \(\mathcal G(P)\cong\Gamma(\operatorname{Ad}(P))\). We obtain a necessary and sufficient lifting criterion in terms of the trivializability of this factor set, and show that, once a single lift exists, the set of all lifts modulo gauge conjugacy is classified by a pointed nonabelian \(H^1\)-set.

We then apply this global theory to local quaternionic torus actions. Pulling back to the universal covering of the orbit space, untwists a local \(Q^n\)-action and produces a globally defined action on the pulled-back manifold. A preliminary lift of this global action need not be compatible with the deck transformations. We define a gauge-valued nonabelian descent defect, establish its crossed-cocycle identities and transformation law, and prove that the original local action lifts if and only if the global lifting obstruction vanishes and the descent defect is trivializable. In the abelian case, these constructions reduce to the classical obstruction theory for lifts of local torus actions.
\end{abstract}

% Keywords and MSC
\noindent \keywords{quaternionic torus action, lifting group action, principal \text{Sp}(1)-bundle, nonabelian descent obstruction, equivariant lift.} \\[1ex]
\msc[2020] {Primary 57S15, 57S25; Secondary 55R10, 55R91, 55S35.}

\vspace{0.5ex}
\noindent \textbf{Conflict of Interest and Data Availability Statement:} The authors state that there is no conflict of interest to declare. Moreover, data sharing is not applicable to this article as no datasets were generated or analyzed during the research carried out in this paper.

\tableofcontents

\section{Introduction}

The problem of lifting group actions on manifolds to compatible group actions on fiber bundles is a classical question in transformation group theory and equivariant topology. Given a topological group \(G\) acting on a space \(X\) and a principal bundle \[ H\longrightarrow P\overset{\pi_P}{\longrightarrow}X, \] one asks whether the action of \(G\) on \(X\) can be lifted to an action on \(P\) by principal \(H\)-bundle automorphisms. Such a lift is not merely a collection of bundle isomorphisms covering the transformations of \(X\) itself: these isomorphisms must depend continuously on \(G\) and satisfy the group law.
 Early results on this problem were obtained by Stewart~\cite{steward}, who studied lifts of compact group actions in fiber bundles and obtained, in particular, lifting results for actions of simply connected semisimple compact Lie groups. Hattori and Yoshida~\cite{hatyos} subsequently developed a systematic obstruction theory for lifting actions to principal torus bundles. More precisely, let \(G\) be a compact Lie group acting on a connected, locally finite CW complex \(X\), and let
\[
T^n\longrightarrow P\longrightarrow X
\]
be a principal torus bundle. The Borel construction of the \(G\)-space \(X\) is
\[
X_G:=EG\times_G X,
\]
where \(EG\to BG\) is a universal principal \(G\)-bundle. The inclusion of a fiber $j:X\longrightarrow X_G$ induces a restriction map from principal \(T^n\)-bundles over \(X_G\) to principal \(T^n\)-bundles over \(X\). Hattori and Yoshida prove that the \(G\)-action on \(X\) lifts to \(P\) if and only if the isomorphism class of \(P\) lies in the image of this restriction map. Since principal \(T^n\)-bundles are classified by their first Chern classes, this is equivalent to requiring that $c_1(P)\in H^2(X;\mathbb Z^n)$ admits an equivariant extension, namely a class in $H_G^2(X;\mathbb Z^n)=H^2(X_G;\mathbb Z^n)$ whose restriction to \(X\) is \(c_1(P)\). Related classification results for equivariant bundles with abelian structural group were developed by Lashof--May--Segal~\cite{lms}.

The abelian nature of the structural torus plays an essential role in the classical theory. For a principal \(T^n\)-bundle, the gauge group is naturally identified with the abelian group \(C(X,T^n)\) of continuous maps from \(X\) to \(T^n\). Thus, after choosing a continuous section of the extension of bundle automorphisms covering the \(G\)-action, the failure of this section to satisfy the group law is measured by an ordinary continuous group-cohomology class. In this context, a \emph{pseudo-lift} means only such a continuous section, whereas a \emph{lift} is a pseudo-lift which is also a homomorphism. When a lift exists, the set of lifts modulo gauge conjugacy is described by an ordinary first cohomology group.

The purpose of the present work is to develop an analogous theory for principal bundles whose structural group is a quaternionic torus, i.e. $H=Q^k=(\operatorname{Sp}(1))^k\cong (SU(2))^n\cong (S^3)^n.$ As \(Q^k\) is nonabelian, the gauge group 
\[ 
\mathcal G(P) \cong \Gamma\bigl(\operatorname{Ad}(P)\bigr) 
\] 
is generally nonabelian, and the ordinary cohomological formulation of the torus case does not carry over directly. This changes both the nature of the global obstruction and the classification of lifts. 

We first consider the global problem. Let \(G\) be a compact Lie group acting continuously on a connected, locally finite CW complex \(X\), and let 
\[ Q^k\longrightarrow P\overset{\pi_P}{\longrightarrow}X \] 
be a principal quaternionic torus bundle. A \emph{lift} (see Definition~\ref{def:global-lift}) of the \(G\)-action to \(P\) is a continuous action of \(G\) on \(P\) by principal \(Q^k\)-bundle automorphisms, that is, 
\[ \pi_P(g\cdot p)=g\cdot\pi_P(p), \qquad g\cdot(p\cdot q)=(g\cdot p)\cdot q,\;\;\;\forall g\in G,\;p\in P,\;q\in Q^k. \] 

\noindent Let \(\mathcal Q^k(X)\) denote the set of isomorphism classes of principal \(Q^k\)-bundles over \(X\) and denote by \[ \mathcal Q^k_G(X)\subseteq \mathcal Q^k(X) \] 
the subset of isomorphism classes of principal \(Q^k\)-bundles over \(X\) which admit such a lifted \(G\)-action.

\noindent Let \(EG\to BG\) be a universal principal \(G\)-bundle. The Borel construction of the \(G\)-space \(X\) is \[ X_G:=EG\times_G X.\] 
The inclusion of a fiber $j:X\longrightarrow X_G $ induces a restriction map \[ j^*:\mathcal{Q}^k(X_G)\longrightarrow \mathcal{Q}^k(X), \] where \(\mathcal{Q}^k(Y)\) denotes the set of isomorphism classes of principal \(Q^k\)-bundles over a space \(Y\). We write \[ \mathcal Q^k_{\mathrm{Bor}}(X;G) := \operatorname{im}(j^*) \subseteq \mathcal Q^k(X) \] for the corresponding Borel-image subset. 

If the \(G\)-action on \(X\) lifts to \(P\), then \(P\) becomes a \(G\)-equivariant principal \(Q^k\)-bundle. Hence its Borel construction \[ P_G:=EG\times_G P \] is a principal \(Q^k\)-bundle over \(X_G\), and the restriction of \(P_G\) along \(j:X\to X_G\) is naturally isomorphic to \(P\). Therefore \[ \mathcal Q^k_G(X)\subseteq \mathcal Q^k_{\mathrm{Bor}}(X;G). \] 

\noindent This necessary Borel condition is proved in Proposition~\ref{prop:lifting-implies-borel}. Throughout this paper, ordinary cohomology means singular cohomology with the indicated coefficients, and equivariant cohomology is understood in the Borel sense, \[ H^*_G(X;A):=H^*(X_G;A). \] 
Thus, as a consequence from the discussion above, the second Chern class $c_2(P)\in H^4(X;\mathbb Z^k)$ must lie in the image of the forgetful map \[ j^*:H^4_G(X;\mathbb Z^k)=H^4(X_G;\mathbb Z^k) \longrightarrow H^4(X;\mathbb Z^k). \] This condition says precisely that \(c_2(P)\) admits an equivariant extension (see Corollary~\ref{cor:c2-obstruction-global}). We emphasize that \(c_2(P)\) is used only as the first characteristic datum of the bundle. It is not, in general, a complete classification invariant for principal \((\operatorname{Sp}(1))^k\)-bundles. Already for quaternionic line bundles over quaternionic projective spaces, unstable homotopy-theoretic data beyond \(c_2\) may enter the classification~\cite{granja}; moreover, when \(c_2\) is torsion, a secondary invariant is needed to complete the classification~\cite{crowleygoette}.

To describe the remaining obstruction, we consider the extension of bundle automorphisms associated with the \(G\)-action. Let
\[
\mathcal G(P):=\operatorname{Aut}_X(P)
\]
be the gauge group of \(P\), namely the group of principal \(Q^k\)-bundle automorphisms covering the identity map of \(X\). Let  \(\varphi:G\to\operatorname{Homeo}(X)\) be the given action on \(X\) and define
\[
\widehat{\operatorname{Isom}}_G(P)
:=
\left\{
(g,F)\in G\times\operatorname{Isom}(P):
\varphi(g)\circ\pi_P=\pi_P\circ F
\right\}.
\]
Thus an element of \(\widehat{\operatorname{Isom}}_G(P)\) is a bundle automorphism of \(P\) covering some transformation of \(X\) coming from the \(G\)-action. Together with the natural projection to the first factor, this builds an exact sequence
\[
1\longrightarrow
\mathcal G(P)
\longrightarrow
\widehat{\operatorname{Isom}}_G(P)
\longrightarrow
G
\longrightarrow
1,
\]
provided every \(g\in G\) is covered by at least one principal bundle
automorphism of \(P\).

\noindent We then define a \emph{pseudo-lift} of the \(G\)-action to \(P\) to be a continuous section
\[
s:G\longrightarrow\widehat{\operatorname{Isom}}_G(P)
\]
of this projection. Thus \(s(g)\) is a principal bundle automorphism of \(P\) covering \(\varphi(g)\), but note that \(s\) is not required to satisfy the group law. Thus we furthermore define a \emph{lift} to be precisely a pseudo-lift that  is a homomorphism.

Given a pseudo-lift \(s\), define the associated \emph{factor set} (see Definition~\ref{def:factor-set}), as a map,
\[
\omega_s:G\times G\longrightarrow\mathcal G(P),
\]
such that
\[
\omega_s(g,h)=s(g)s(h)\bigl(s(gh)\bigr)^{-1}.
\]
This takes values in \(\mathcal G(P)\), because both \(s(g)s(h)\) and \(s(gh)\) cover the same homeomorphism \(\varphi(gh)\) of \(X\). Hence \(\omega_s\) measures exactly the failure of \(s\) to be a homomorphism:
\[
\omega_s=1
\quad\Longleftrightarrow\quad
s(g)s(h)=s(gh)\ \text{for all }g,h\in G.
\]

In turn, associativity in \(\widehat{\operatorname{Isom}}_G(P)\) implies the nonabelian factor-set identity
\[
\omega_s(g,h)\,\omega_s(gh,\ell)
=
{}^{s(g)}\omega_s(h,\ell)\,\omega_s(g,h\ell),
\]
where
\[
{}^{s(g)}\gamma:=s(g)\gamma s(g)^{-1},
\qquad
\gamma\in\mathcal G(P).
\]
Thus the factor set satisfies the usual cocycle identity, but with a nontrivial conjugation action on the nonabelian gauge group.

To introduce the last notion before the statement of our first main result, let \(a:G\to\mathcal G(P)\) be a continuous gauge-valued map. Then
\[
s^a(g):=a(g)s(g)
\]
is another pseudo-lift. Its factor set is
\[
\omega_{s^a}(g,h)
=
a(g)\,{}^{s(g)}a(h)\,\omega_s(g,h)\,a(gh)^{-1}.
\]
We say (see Definition~\ref{def:factor-set-trivializable}) that \(\omega_s\) is \emph{trivializable} if there exists such a continuous map \(a:G\to\mathcal G(P)\) for which
\[
\omega_{s^a}(g,h)=1,\;\;\forall g,h\in G.
\]
Equivalently, the modified pseudo-lift \(s^a\) is a lift. This leads to the following criterion.
\begin{theorem}[Global lifting criterion; Theorem~\ref{thm:global-q-HY}] \label{thm:intro-global} 
Let \(G\) be a compact Lie group acting continuously on a connected, locally finite CW complex \(X\), and let \( Q^k\longrightarrow P\longrightarrow X \) be a principal \(Q^k\)-bundle. Assume that the class of \(P\), as an isomorphism class of principal \(Q^k\)-bundles over \(X\), belongs to the Borel-image subset \( \mathcal Q^k_{\mathrm{Bor}}(X;G) \subseteq \mathcal Q^k(X), \) so that the projection \( \widehat{\operatorname{Isom}}_G(P)\longrightarrow G \) admits a continuous section. Then the \(G\)-action on \(X\) lifts to \(P\) if and only if the factor set \( \omega_s(g,h) = s(g)s(h)\bigl(s(gh)\bigr)^{-1} \in \mathcal G(P) \) associated with one, and hence every, pseudo-lift \( s:G\longrightarrow\widehat{\operatorname{Isom}}_G(P) \) is trivializable. 
\end{theorem}

 The theorem separates the global lifting problem into two explicit steps. First, the underlying principal \(Q^k\)-bundle must come from the Borel
construction, in the sense that
\[
[P]\in\mathcal Q^k_{\mathrm{Bor}}(X;G)\subseteq\mathcal Q^k(X).
\]
This is the homotopy-theoretic condition which gives the existence of bundle isomorphisms covering the individual transformations of \(X\). Second, these bundle isomorphisms must be chosen so as to satisfy the group law. The obstruction to this second step is the gauge-valued factor set \(\omega_s\). The action lifts precisely when this factor set can be killed by modifying the pseudo-lift by a continuous gauge-valued map. Thus we do not need to construct a full nonabelian second cohomology set; we only use the well-defined property that \(\omega_s\) is trivializable.

We also describe all lifts once one lift has been fixed. Let
\[
\widetilde\phi:G\longrightarrow\widehat{\operatorname{Isom}}_G(P)
\]
be a lift. Any other lift \(\widetilde\phi'\) covers the same \(G\)-action on \(X\), and therefore there is a unique continuous map $a:G\longrightarrow\mathcal G(P)$
such that
$\widetilde\phi'(g)=a(g)\widetilde\phi(g)$.
The condition that \(\widetilde\phi'\) is again a homomorphism is equivalent to the crossed homomorphism identity
\[
a(gh)=a(g)\,{}^{\widetilde\phi(g)}a(h),
\]
where
\[
{}^{\widetilde\phi(g)}\gamma
:=
\widetilde\phi(g)\gamma\widetilde\phi(g)^{-1},\;\;
\gamma\in\mathcal G(P).
\]
Thus \(a\) is a continuous nonabelian \(1\)-cocycle for the conjugation action of \(G\) on \(\mathcal G(P)\) induced by the fixed lift
\(\widetilde\phi\). Denote by \[H^1_{\mathrm{nab}} \bigl(G;\mathcal G(P)_{\widetilde\phi}\bigr)\]
 the pointed nonabelian cohomology set of such continuous crossed homomorphisms $a:G\longrightarrow\mathcal G(P)$ with the following equivalence relation: Two lifts are equivalent if they differ by conjugation by a single gauge transformation. More explicitly, if \(b\in\mathcal G(P)\), then conjugating \(\widetilde\phi'\) by \(b\) replaces \(a\) by
\[
a^b(g)
=
b\,a(g)\,{}^{\widetilde\phi(g)}b^{-1}.
\]

\noindent The distinguished point is the class of the trivial cocycle \(a(g)=1\) corresponding to the fixed lift \(\widetilde\phi\). \\
We then give the following enumeration result.

\begin{theorem}[Enumeration of lifts; Theorem~\ref{thm:enumeration-global-liftings}] \label{thm:intro-enumeration} If the \(G\)-action on $X$ admits a lift to \(P\), then the set of lifts modulo gauge conjugacy is naturally identified with the pointed nonabelian set $H^1_{\mathrm{nab}} \bigl(G;\mathcal G(P)_{\widetilde\phi}\bigr)$. 
\end{theorem} 

The second part of the paper concerns local quaternionic torus actions. Yoshida~\cite{Yoshida2,Yoshida1} introduced local torus actions modeled on the coordinatewise action of \(T^n\) on \(\mathbb C^n\) and studied their topology and their liftings to principal bundles. In our previous work \cite{bg25,bg_top}, we developed a quaternionic counterpart modeled on the coordinatewise action of \(Q^n\) on \( \mathbb H^n\).

Let \(M\) be a smooth \(4n\)-dimensional manifold equipped with a local \(Q^n\)-action \(\mathcal Q\). Such an action is represented by a weakly regular atlas whose transition maps are equivariant up to automorphisms of \(Q^n\). The corresponding automorphisms determine a \v{C}ech class, or equivalently a monodromy representation
\[
\rho:\pi_1(B_M)\longrightarrow\operatorname{Aut}(Q^n),
\]
where \(B_M\) is the orbit space of the local quaternionic torus action on \(M\), and \(\pi_M:M\longrightarrow B_M\) is the orbit map.

Passing to the universal covering \(p:\widetilde B_M\longrightarrow B_M\) trivializes the monodromy class. We define the pulled-back space \(\widetilde M:=p^*M\) as the fiber product
\[
\widetilde M=\left\{ (\widetilde b,x)\in \widetilde B_M\times M:p(\widetilde b)=\pi_M(x)\right\}.
\]
The projection
\[
\widetilde M\longrightarrow \widetilde B_M,
\qquad
(\widetilde b,x)\longmapsto \widetilde b,
\]
is the pullback of the orbit map \(\pi_M:M\to B_M\).

Since \(\widetilde B_M\) is simply connected, the pulled-back local \(Q^n\)-action has trivial monodromy and therefore is induced by a globally defined locally regular \(Q^n\)-action on \(\widetilde M\). This action combines with the deck transformations of
\[
p:\widetilde B_M\longrightarrow B_M
\]
to give an action of \(Q^n\rtimes_\rho\pi_1(B_M)\) on \(\widetilde M\).

Let now
\[
Q^k\longrightarrow P\overset{\pi_P}{\longrightarrow}M
\]
be a principal quaternionic torus bundle. Its pullback to \(\widetilde M\) is the principal \(Q^k\)-bundle
\[
Q^k\longrightarrow \widetilde P\longrightarrow \widetilde M
\]
defined by
\[
\widetilde P=\left\{ (\widetilde b,p_0)\in \widetilde B_M\times P:p(\widetilde b)=\pi_M(\pi_P(p_0))
\right\}.
\]

The lifting problem now has two stages. First, the globally defined \(Q^n\)-action on \(\widetilde M\) must admit a lift to the pulled-back bundle \(\widetilde P\), in the sense of the global lifting problem above. We call such a lift a \emph{preliminary lift} because it only lifts the \(Q^n\)-part of the untwisted action; it is not yet required to be compatible with the deck transformations. Thus a preliminary lift is a homomorphism \[ L:Q^k\longrightarrow\operatorname{Isom}(\widetilde P) \] covering the global action of \(Q^k\) on \(\widetilde M\). We write $L_u:=L(u),\;\; u\in Q^k$.

Second, the preliminary lift must be compatible with the canonical lift of the deck-transformation action. For \(a\in\pi_1(B_M)\), let \[ D_a:\widetilde P\longrightarrow\widetilde P \] denote the canonical lift of the deck transformation \(a\), given by \[ D_a(\widetilde b,p_0)=(\widetilde b\cdot a^{-1},p_0). \] The compatibility condition is \[ D_a^{-1}L_{\rho(a)(u)}D_a=L_u,\;\;\forall a\in\pi_1(B_M),\;u\in Q^k.\] 

The failure of this condition to hold is measured by the map \[ \Theta_L:\pi_1(B_M)\times Q^k \longrightarrow \mathcal G(\widetilde P) \] defined by \[ \Theta_L(a,u) = D_a^{-1}L_{\rho(a)(u)}D_aL_u^{-1}. \] We call \(\Theta_L\) the \emph{descent defect} of the preliminary lift \(L\) (see Definition~\ref{def:descent-defect}). It takes values in the gauge group because both \(D_a^{-1}L_{\rho(a)(u)}D_a\) and \(L_u\) cover the same homeomorphism of \(\widetilde M\). The equality \[ \Theta_L(a,u)=1\;\;\forall a\in\pi_1(B_M),\;u\in Q^k, \] is precisely the condition that the preliminary lift descends to a lift of the original local \(Q^k\)-action on \(M\).

The map \(\Theta_L\) satisfies nonabelian compatibility identities in the \(Q^k\)-variable and in the \(\pi_1(B_M)\)-variable (see Lemmas~\ref{lem:defect-Q-cocycle} and \ref{lem:defect-Gamma-cocycle}). These identities are the descent analogues of the factor-set identity in the global problem. We also compute how \(\Theta_L\) changes when the preliminary lift is modified by a gauge-valued crossed homomorphism, and we show that its trivializability is independent of the chosen preliminary lift.

We say (see Definition~\ref{def:descent-trivializable}) that the descent defect \(\Theta_L\) is \emph{trivializable} if there exists a continuous gauge-valued crossed homomorphism \[ \tau:Q^n\longrightarrow\mathcal G(\widetilde P), \] such that, after modifying the preliminary lift by \[ L^\tau_u:=\tau(u)L_u, \] the modified lift \(L^\tau\) satisfies the deck-compatibility condition \[ D_a^{-1}L^\tau_{\rho(a)(u)}D_a=L^\tau_u,\;\;\forall a\in\pi_1(B_M),\;u\in Q^k.\] Equivalently, \[ D_a^{-1}\tau\bigl(\rho(a)(u)\bigr)D_a\, \Theta_L(a,u)\, \tau(u)^{-1} = 1,\;\;\forall a\in\pi_1(B_M),\;u\in Q^k. \] 

\begin{theorem}[Local lifting criterion; Theorem~\ref{thm:complete-local-lifting}] \label{thm:intro-local} The local \(Q^n\)-action on a smooth manifold \(M\) lifts to \(P\) if and only if: 
\begin{enumerate} 
\item[(1)] the globalized \(Q^n\)-action on \(\widetilde M\) admits a preliminary lift to \(\widetilde P\); and 
\item[(2)] the associated nonabelian descent defect is trivializable. \end{enumerate} 
\end{theorem} 

In the abelian torus case, the conjugation action of the structural group is trivial, and therefore the gauge group is canonically \[ \mathcal G(P)\cong C(M,T^k). \] Consequently, the global factor set takes values in the abelian group \(C(M,T^k)\) and determines an ordinary continuous group-cohomology class \[ [\omega_s]\in H^2\bigl(G;C(M,T^k)\bigr), \] as in Hattori--Yoshida~\cite{hatyos}. Similarly, in the local torus case, the descent identities become ordinary cocycle identities for the \(\pi_1(B_M)\)-action on the group of \(T^n\)-cocycles with values in \(C(\widetilde M,T^k)\), giving the obstruction group used by Yoshida~\cite{Yoshida1}. Thus, after replacing \(Q^k\) by \(T^k\), the formal identities appearing in the present paper specialize to the cohomological identities of the classical torus theory. In the quaternionic case, these same identities take values in nonabelian gauge groups, and the obstructions are formulated in terms of trivializability rather than ordinary cohomology classes.

\paragraph*{Organization of paper.}The paper is organized as follows. Section~\ref{sec:global-lifting} develops the global lifting obstruction and the enumeration of lifts. Section~\ref{sec:untwisting} recalls local quaternionic torus actions, constructs the untwisted action on the universal covering of the orbit space, and gives the chartwise description of lifts of local actions. Finally, Section~\ref{sec:nonabelian-descent} defines the descent defect and proves the local lifting criterion.

\textbf{Assumptions.}
All spaces are assumed to be Hausdorff, paracompact, and of the homotopy type of CW complexes. In Section~\ref{sec:global-lifting}, \(X\) is a connected, locally finite CW complex with a continuous action of a compact Lie group \(G\). In Sections~\ref{sec:untwisting}--\ref{sec:nonabelian-descent}, \(M\) is a smooth \(4n\)-manifold equipped with a local \(Q^n\)-action, and the orbit space \(B_M\) is assumed to admit a universal covering.

\paragraph*{Acknowledgments.}
The authors are grateful to Georgios Raptis and Stephen Theriault for their careful reading of an earlier version of this work and for valuable comments and suggestions. We thank Georgios Raptis in particular for comments which helped clarify the distinction between Borel-image data, pseudo-lifts, liftability and discussions around Remark~\ref{rem:moduli-future-work}. We also thank Stephen Theriault for drawing our attention to the role of gauge groups over \(S^4\) and related homotopy-theoretic issues, which are discussed in Remark~\ref{rem:gauge-groups-stabilization}.

\section{The global lifting problem}
\label{sec:global-lifting}

In this section we formulate the global lifting problem for principal quaternionic torus bundles. The discussion follows the classical framework of Hattori--Yoshida for principal torus bundles \cite{hatyos}, with the important difference that now the structure group, \(Q^k=(\text{Sp}(1))^k\), is nonabelian.

\subsection{Principal quaternionic torus bundles and lifts of actions}
Let \(G\) be a compact Lie group acting continuously on a connected, locally finite CW complex \(X\). Equivalently, we are given a continuous homomorphism \[ \phi:G\longrightarrow \operatorname{Homeo}(X), \] where \(\operatorname{Homeo}(X)\) is endowed with the compact-open topology. Following \cite{Hat02, michell}, let $EG\longrightarrow BG$
be a universal principal \(G\)-bundle. The Borel construction of \(X\) is
\[
X_G:=EG\times_G X,
\]
and we consider the associated fibration $X\longrightarrow X_G\longrightarrow BG$.
Choosing a point \(e_0\in EG\), the inclusion of the fiber is denoted by
\[
j:X\longrightarrow X_G,
\qquad
x\longmapsto [e_0,x].
\]
Different choices of \(e_0\in EG\) give homotopic inclusions of the fiber, so the induced map \(j^*\) is independent of this choice up to the usual canonical identification.  Since $H_G^*(X;R)=H^*(X_G;R)$, the induced map
\[
j^*:H^*(X_G;R)\longrightarrow H^*(X;R)
\]
is the usual forgetful map from equivariant cohomology of X to ordinary cohomology of $X_G$
(see details in \cite[p.14]{hatyos}).
Following Hopkinson~\cite[Proposition~6.2.3]{ho}, it is $BQ=B\text{Sp}(1)=\mathbb HP^\infty$, thus \(BQ^k\simeq(\mathbb HP^\infty)^k\), and the integral cohomology ring of \(BQ^k\) is
\[
H^*(BQ^k;\mathbb Z)
\cong
\mathbb Z[u_1,\ldots,u_k],
\qquad |u_i|=4.
\]

If $Q^k\longrightarrow P\overset{\pi_P}{\longrightarrow} X$
is a principal \(Q^k\)-bundle with classifying map $f_P:X\longrightarrow BQ^k$,
we define
\[
c_2(P)
=
(c_2^1(P),\ldots,c_2^k(P))
\in H^4(X;\mathbb Z^k)
\]
by \(c_2^i(P)=f_P^*(u_i)\).

\begin{remark}
The class \(c_2(P)\) is the first characteristic datum of a principal \(Q^k\)-bundle. It should not be understood as a complete classification invariant in general, since the classifying space \(BQ^k\) contains higher homotopy information beyond its first nontrivial degree--\(4\) cohomology classes. In this paper \(c_2(P)\) is used only as a primary obstruction to equivariance.
\end{remark}

\noindent Throughout the text, let $\mathcal{Q}^k(Y)$
denote the set of isomorphism classes of principal \(Q^k\)-bundles over a given space \(Y\).

\begin{definition}\label{def:global-lift}
Let \(G\) be a compact Lie group acting continuously on \(X\). We denote by
\[
\mathcal Q^k_G(X)\subseteq \mathcal Q^k(X)
\]
the subset consisting of those isomorphism classes represented by principal \(Q^k\)-bundles
\[
Q^k\longrightarrow P\overset{\pi_P}{\longrightarrow}X
\] for which the \(G\)-action on \(X\) admits a lift to \(P\). Explicitly, this means that there exists a continuous action
\[
G\times P\longrightarrow P,\qquad (g,p)\longmapsto g\cdot p,
\]
such that
\[
\pi_P(g\cdot p)=g\cdot \pi_P(p),\;\;
g\cdot(p\cdot q)=(g\cdot p)\cdot q,\;\;\;\forall g\in G,p\in P,q\in Q^k.
\]
 Equivalently, the action of \(G\) on \(P\) is by principal \(Q^k\)-bundle automorphisms covering the given action of \(G\) on \(X\).
\end{definition}

\noindent Next we define the Borel-image subset. The inclusion of the fiber of the Borel construction of $X$, the map $j:X\longrightarrow X_G$,
naturally induces a restriction map
\[
j^*:\mathcal{Q}^k(X_G)\longrightarrow \mathcal{Q}^k(X).
\]
\begin{definition}
We define the \emph{Borel-image subset} \(\mathcal Q^k_{\mathrm{Bor}}(X;G) \subseteq \mathcal Q^k(X)\) by\[\mathcal Q^k_{\mathrm{Bor}}(X;G) := \operatorname{im}\left(j^*:\mathcal{Q}^k(X_G)\longrightarrow\mathcal{Q}^k(X)
\right).\] 
\end{definition} 
Thus an isomorphism class $[P]\in\mathcal Q^k(X)$
belongs to \(\mathcal Q^k_{\mathrm{Bor}}(X;G)\) if and only if there exists a
principal \(Q^k\)-bundle
\[
Q^k\longrightarrow R\longrightarrow X_G
\]
such that \(P\cong j^*R
\) as principal \(Q^k\)-bundles over \(X\). This is the quaternionic analogue of the Borel-image subset used by Hattori--Yoshida for principal torus bundles.

\begin{prop}
\label{prop:lifting-implies-borel}
Let \(G\) be a compact Lie group acting on \(X\). If a principal \(Q^k\)-bundle $Q^k\longrightarrow P\longrightarrow X$
admits a lift of the \(G\)-action, then
\[
[P]\in \mathcal Q^k_{\mathrm{Bor}}(X;G).
\]
In other words, $\mathcal Q^k_G(X)\subseteq \mathcal Q^k_{\mathrm{Bor}}(X;G).$
\end{prop}
 \begin{proof}
If the \(G\)-action lifts to \(P\), then, by Definition \ref{def:global-lift}, \(P\) is a \(G\)-equivariant principal \(Q^k\)-bundle; the \(G\)-action on \(P\) covers the given \(G\)-action on \(X\) and commutes with the principal right \(Q^k\)-action. Therefore the Borel construction
\[
P_G:=EG\times_G P
\]
is a principal \(Q^k\)-bundle over $X_G:=EG\times_G X$.

Let again $j:X\longrightarrow X_G$
be the inclusion of the fiber, determined by a choice of point \(e_0\in EG\). The restriction of \(P_G\) along \(j\) is naturally isomorphic to \(P\),
\[
j^*P_G\cong P.
\]
Indeed, under this identification the point \(p\in P\) corresponds to the class \([e_0,p]\in EG\times_G P\). Hence the isomorphism class \([P]\in\mathcal Q^k(X)\) lies in the image of
the restriction map $j^*:\mathcal{Q}^k(X_G)
\longrightarrow
\mathcal{Q}^k(X)$.
\end{proof}

\noindent We have thus established the inclusions
\[
\mathcal Q^k_G(X)\subseteq \mathcal Q^k_{\mathrm{Bor}}(X;G)\subseteq \mathcal Q^k(X).
\]
The second inclusion simply means that \(\mathcal Q^k_{\mathrm{Bor}}(X;G)\) is defined as a subset of \(\mathcal Q^k(X)\) by restriction along the fiber inclusion \(j:X\to X_G\). Proposition~\ref{prop:lifting-implies-borel} says that a genuine lift of the \(G\)-action forces the underlying principal \(Q^k\)-bundle to come from the Borel construction.

\noindent Applying the characteristic class \(c_2\) to this Borel condition gives a necessary cohomological condition, as the next corollary explains.

\begin{cor}[Equivariant \(c_2\)-obstruction]
\label{cor:c2-obstruction-global}
If the \(G\)-action on \(X\) lifts to \(P\), then
\[
c_2(P)\in H^4(X;\mathbb Z^k)
\]
lies in the image of the forgetful map
\[
H_G^4(X;\mathbb Z^k)
=
H^4(X_G;\mathbb Z^k)
\longrightarrow
H^4(X;\mathbb Z^k).
\]
Equivalently, \(c_2(P)\) admits an equivariant extension.
\end{cor}

\begin{proof}
By Proposition~\ref{prop:lifting-implies-borel}, there exists a principal \(Q^k\)-bundle $R\longrightarrow X_G$
such that $P\cong j^*R$.
Define
\[
c_2^G(P):=c_2(R)\in H^4(X_G;\mathbb Z^k)=H_G^4(X;\mathbb Z^k).
\]
By naturality of Chern classes, $j^*c_2^G(P)=c_2(P)$.
\end{proof}

\begin{remark}
For comparison, let \(\mathcal T^n(X)\) denote the set of isomorphism classes of principal \(T^n\)-bundles over \(X\), and let \[ \mathcal T^n_G(X)\subseteq \mathcal T^n(X) \] denote the subset consisting of those classes which admit a lift of the \(G\)-action. Similarly, define the Borel-image subset \[\mathcal T^n_{\mathrm{Bor}}(X;G):=\operatorname{im}\left( j^*:\mathcal T^n(X_G)\longrightarrow \mathcal T^n(X)\right)\subseteq\mathcal T^n(X).\] Hattori--Yoshida~\cite[Theorem~1.1]{hatyos} prove that \[ \mathcal T^n_G(X) = \mathcal T^n_{\mathrm{Bor}}(X;G). \]
 Since principal \(T^n\)-bundles are classified by their first Chern classes, this is equivalent to the statement that a lift of the $G-$ action exists if and only if the Chern class of the principal bundle admits an equivariant extension.

Back to principal \(Q^k\)-bundles, Proposition~\ref{prop:lifting-implies-borel} gives the necessary direction. The converse is not automatic, because the gauge group of a principal \(Q^k\)-bundle is generally nonabelian. This is the first difference in the quaternionic lifting problem.
\end{remark}

\subsection{The group extension associated to a bundle}

The purpose of this subsection is to isolate the obstruction-theoretic object attached to the bundle \(P\). A lift of the \(G\)-action is the same thing as a continuous splitting homomorphism of a natural extension whose kernel is the gauge group of \(P\). We first construct this extension and then use it in the next subsection to define pseudo-liftings and their factor sets.

Let $Q^k\longrightarrow P\overset{\pi_P}{\longrightarrow}X$ be a principal \(Q^k\)-bundle over $X$. Let $\operatorname{Isom}(P)$ denote the group of principal bundle automorphisms of \(P\), endowed with the compact-open topology. Since \(X\) is assumed to be a locally finite CW complex, it is locally compact, and under our standing assumptions the compact-open topology makes \(\operatorname{Isom}(P)\) a topological group. Every element of \(\operatorname{Isom}(P)\) induces a homeomorphism of \(X\). We write
\[
q:\operatorname{Isom}(P)\longrightarrow \operatorname{Homeo}(X),
\]
for the resulting homomorphism. Given an action $\phi:G\longrightarrow \operatorname{Homeo}(X)$,
define
\[
\widehat{\operatorname{Isom}}_G(P)
:=
\{(g,F)\in G\times \operatorname{Isom}(P):\phi(g)=q(F)\}.
\]
Let $\mathcal G(P):=\operatorname{Aut}_X(P)$
be the gauge group of \(P\), that is the subgroup of \(\operatorname{Isom}(P)\) consisting of principal \(Q^k\)-bundle automorphisms covering the identity map of \(X\) (cf. \cite[Chapter 7]{Husemoller}),
\[
\mathcal G(P)=\{F\in\operatorname{Isom}(P):\pi_P\circ F=\pi_P\}.\]
Let also \[
G_P:=
\{g\in G: \phi(g)^*P\cong P\},
\]
be the subgroup of \(G\) consisting of those elements whose action on \(X\) preserves the isomorphism class of \(P\). 
There is always a short exact sequence
\[
1\longrightarrow \mathcal G(P)
\longrightarrow
\widehat{\operatorname{Isom}}_G(P)
\longrightarrow
G_P
\longrightarrow 1,
\]
and in particular, the obvious projection
\[
\widehat{\operatorname{Isom}}_G(P)\longrightarrow G_P
\]
is surjective if and only if \(G_P=G\), or, equivalently, if and only if $\phi(g)^*P\cong P, \;\forall g\in G$. It is then easy to see that a lift of the \(G\)-action to \(P\) is precisely a continuous splitting homomorphism $\widetilde\phi:G\longrightarrow \widehat{\operatorname{Isom}}_G(P)$ of this extension.

\iffalse
The gauge group has the standard description
\[
\mathcal G(P)\cong \Gamma(\operatorname{Ad}(P)),
\]
where
\[
\operatorname{Ad}(P)=P\times_{Q^k}Q^k
\]
and \(Q^k\) acts on itself by conjugation. Since \(Q^k\) is nonabelian, the gauge group \(\mathcal G(P)\) is generally nonabelian.
\fi

\begin{lemma}
\label{lem:borel-implies-invariance}
If $[P]\in \mathcal Q^k_{\mathrm{Bor}}(X;G)$,
then \(G_P=G\).
\end{lemma}

\begin{proof}
By assumption, there exists a principal \(Q^k\)-bundle $R\longrightarrow X_G$
such that $P\cong j^*R$.
Let \(g\in G\). We need to show that $\phi(g)^*P\cong P$.
The maps
\[
j:X\longrightarrow X_G,
\qquad x\longmapsto [e_0,x],
\]
and
\[
j_g:X\longrightarrow X_G, \qquad x\longmapsto [e_0,\phi(g)(x)]
\]
are homotopic as inclusions of fibers of the fibration $X\longrightarrow X_G\longrightarrow BG$.
Indeed, since \(EG\) is contractible, the points \(e_0\) and \(e_0g^{-1}\) are joined by a path in \(EG\), and this path gives a homotopy between \(j_g\) and
\(j\). Therefore,
\[
\phi(g)^*P \cong \phi(g)^*j^*R \cong j_g^*R \cong j^*R \cong P,
\]
hence \(g\in G_P\).
\end{proof}

\subsection{Pseudo-lifts and the factor set}
 When \(G_P=G\), every element of \(G\) is covered by at least one principal \(Q^k\)-bundle automorphism of \(P\), and we may regard \(\widehat{\operatorname{Isom}}_G(P)\longrightarrow G\) as the projection of the automorphism extension associated with \(P\). We start this subsection with two necessary definitions.

\begin{definition}
\label{def:pseudo-lift}
Assume \(G_P=G\). A \emph{pseudo-lift} of the \(G\)-action to \(P\) is a continuous section
\[
s:G\longrightarrow \widehat{\operatorname{Isom}}_G(P)
\]
of this projection.
\end{definition}

\begin{remark}
\label{rem:lifts-as-homomorphic-pseudolifts}
  By the previous definition, \(s(g)\) is a principal \(Q^k\)-bundle automorphism of \(P\) covering \(\phi(g)\), but \(s\) is not required to satisfy the group law. A lift in the sense of Definition~\ref{def:global-lift} is equivalently a pseudo-lift $\widetilde\phi:G\longrightarrow \widehat{\operatorname{Isom}}_G(P)$
which is a group homomorphism. Indeed, a lifted \(G\)-action on \(P\) assigns to each \(g\in G\) a principal \(Q^k\)-bundle automorphism \(\widetilde\phi(g)\) covering \(\phi(g)\), and the action law is precisely the identity $\widetilde\phi(gh)=\widetilde\phi(g)\widetilde\phi(h)$
for all \(g,h\in G\). Conversely, any homomorphic pseudo-lift defines a lifted \(G\)-action on \(P\) by $g\cdot p:=\widetilde\phi(g)(p)$.
\end{remark}

\noindent Recall that the kernel of $\widehat{\operatorname{Isom}}_G(P)\longrightarrow G$
is the gauge group $\mathcal G(P)=\operatorname{Aut}_X(P)
=
\{F\in\operatorname{Isom}(P):\pi_P\circ F=\pi_P\}$.
Equivalently,
\[
\mathcal G(P)\cong \Gamma(\operatorname{Ad}(P)),
\]
where
\[
\operatorname{Ad}(P)=P\times_{Q^k}Q^k
\]
and \(Q^k\) acts on itself by conjugation. Since \(Q^k\) is nonabelian, this gauge group is generally nonabelian.

\begin{definition}\label{def:factor-set}
Given a pseudo-lift \(s\), consider the map
\[
\omega_s:G\times G\longrightarrow \mathcal G(P)
\]
defined by
\[
\omega_s(g,h)
=
s(g)s(h)\bigl(s(gh)\bigr)^{-1}.
\]
The map $\omega_s$ is called the \emph{factor set} associated to the pseudo-lift \(s\).
\end{definition}

Indeed, \(\omega_s(g,h)\) belongs to the gauge group because both \(s(g)s(h)\) and \(s(gh)\) cover the same homeomorphism \(\phi(gh)\) of \(X\). Therefore their quotient \(s(g)s(h)\bigl(s(gh)\bigr)^{-1}\) covers the identity map of \(X\), and hence lies in \(\mathcal G(P)\).

The factor set \(\omega_s\) measures the failure of \(s\) to be a homomorphism. In particular, \(\omega_s=1\) if and only if \(s(g)s(h)=s(gh) \) for every \(g,h\in G\), that is, if and only if \(s\) is a lift in the sense of Definition~\ref{def:pseudo-lift}.

Furthermore, associativity in \(\widehat{\operatorname{Isom}}_G(P)\) implies the nonabelian factor-set identity

\[
\omega_s(g,h)\,\omega_s(gh,\ell)
=
{}^{s(g)}\!\bigl(\omega_s(h,\ell)\bigr)\,
\omega_s(g,h\ell).
\]
where we have set
\[
{}^{s(g)}\gamma
:=
s(g)\gamma s(g)^{-1},
\qquad
\gamma\in\mathcal G(P).
\]

Observe that if \(a:G\to \mathcal G(P)\) is a continuous map, then
\[
s^a(g):=a(g)s(g)
\]
is another pseudo-lift. Its factor set is
\[
\omega_{s^a}(g,h)
=
a(g)\,{}^{s(g)}a(h)\,\omega_s(g,h)\,a(gh)^{-1}.
\]

\begin{definition}\label{def:factor-set-trivializable}
We say that the factor set \(\omega_s\) is \emph{trivializable} if there exists a continuous map $a:G\to \mathcal G(P)$
such that $\omega_{s^a}(g,h)=1$
for all \(g,h\in G\).
\end{definition}

\begin{lemma}
\label{lem:trivializability-independent}
The trivializability of the factor set is independent of the choice of pseudo-lift.
\end{lemma}

\begin{proof}
Let
\[
s,s':G\longrightarrow \widehat{\operatorname{Isom}}_G(P)
\]
be two pseudo-lifts. Since both \(s(g)\) and \(s'(g)\) cover the same homeomorphism \(\phi(g)\) of \(X\), there is a unique element
\[
a(g)\in \mathcal G(P)
\]
such that
\[
s'(g)=a(g)s(g).
\]
The map \(a:G\to\mathcal G(P)\) is continuous because \(s\) and \(s'\) are
continuous. The corresponding factor sets satisfy
\[
\omega_{s'}(g,h)
=
a(g)\,{}^{s(g)}a(h)\,\omega_s(g,h)\,a(gh)^{-1}.
\]
Thus \(s'\) is obtained from \(s\) by a gauge-valued \(1\)-cochain. Consequently, if \(\omega_s\) can be trivialized by some continuous map
\[
b:G\to\mathcal G(P),
\]
then \(\omega_{s'}\) can be trivialized by composing the corresponding gauge modifications. Conversely, the same argument applied to
\[
s(g)=a(g)^{-1}s'(g)
\]
shows that trivializability of \(\omega_{s'}\) implies trivializability of \(\omega_s\). 
\end{proof}

\begin{remark}
The identity
\[
\omega_s(g,h)\,\omega_s(gh,\ell)
=
{}^{s(g)}\!\bigl(\omega_s(h,\ell)\bigr)\,
\omega_s(g,h\ell)
\]
is the nonabelian analogue of the ordinary \(2\)-cocycle condition in group cohomology. Indeed, in the torus case the gauge group $\mathcal{G}(P)$ is canonically $C(X,T^n)$, the group of continuous maps, 
which is abelian. The conjugation action is written as a \(G\)-module action, and so the factor set
\[
\omega_s:G\times G\longrightarrow \mathcal G(P)
\]
satisfies the usual cocycle identity for continuous group cohomology with coefficients in \(\mathcal G(P)\). Therefore it determines an ordinary group cohomology class $o(P)\in H^2(G;C(X,T^n))$.
This is the obstruction class used by Hattori--Yoshida.
 In the present paper we do not need to construct a full nonabelian cohomology set; we only use the well-defined notion of whether the factor set is trivializable.
\end{remark}

\subsection{The global quaternionic obstruction}

The previous subsection associated an extension
\[
1\longrightarrow \mathcal G(P)
\longrightarrow
\widehat{\operatorname{Isom}}_G(P)
\longrightarrow
G
\longrightarrow 1,
\]
to \(P\), whenever the isomorphism class of \(P\) is preserved by the \(G\)-action. In this subsection we explain how the Borel-image condition produces continuous sections of this extension, i.e. pseudo-lifts. We then use the factor set of such a pseudo-lift to formulate the remaining obstruction for a pseudo-lift to be a lift, following Hattori--Yoshida \cite[Lemma~2.3]{hatyos}.

\begin{prop}
\label{prop:pseudolift-from-borel}
Suppose that $[P]\in \mathcal{Q}^k_\mathrm{Bor}(X;G)$. Then \(P\) admits a pseudo-lift of the \(G\)-action.
\end{prop}

\begin{proof}
By assumption, there exists a principal \(Q^k\)-bundle
\[
Q^k\longrightarrow R\longrightarrow X_G
\]
such that \(P\cong j^*R\). Let
\[
\Pi:EG\times X\longrightarrow X_G
\]
be the quotient map, and let
\[
\Pi^*R\longrightarrow EG\times X
\]
be the pullback principal \(Q^k\)-bundle. Fix \(e_0\in EG\). The fiber
inclusion
\[
j:X\longrightarrow X_G,\qquad x\longmapsto [e_0,x],
\]
identifies \(P\cong \Pi^*R\big|_{\{e_0\}\times X}\). Since \(\Pi\) is constant on diagonal \(G\)-orbits in \(EG\times X\), the pullback bundle \(\Pi^*R\) carries a canonical \(G\)-action. Explicitly, if an element of \(\Pi^*R\) is written as a pair
\[
(z,q)\in (EG\times X)\times R, \qquad \pi_R(q)=\Pi(z),
\]
then we define \(g\cdot(z,q):=(gz,q)\). This is well defined because \(\Pi(gz)=\Pi(z)\). It defines a continuous action of \(G\) on \(\Pi^*R\) by principal \(Q^k\)-bundle automorphisms covering the diagonal action of \(G\) on \(EG\times X\). Moreover, it commutes with the principal right \(Q^k\)-action.

Let
\[
r_t:EG\longrightarrow EG,\qquad t\in[0,1],
\]
be a contraction of \(EG\) to \(e_0\), with
\[
r_1=\operatorname{id}_{EG}, \qquad r_0\equiv e_0.
\]
Consider the homotopy
\[
r_t\times\operatorname{id}_X:EG\times X\longrightarrow EG\times X.
\]
Since \(EG\times X\) has the homotopy type of a CW complex and our bundles are principal bundles over paracompact spaces, the covering homotopy theorem for principal bundles applies. Starting from the identity bundle map \(\widetilde r_1=\operatorname{id}_{\Pi^*R}\) covering \(r_1\times\operatorname{id}_X=\operatorname{id}_{EG\times X}\), we obtain a continuous family of principal \(Q^k\)-bundle maps \(\widetilde r_t:\Pi^*R\longrightarrow \Pi^*R \) covering \(r_t\times\operatorname{id}_X.\) In particular,
\(\widetilde r_0:\Pi^*R\longrightarrow \Pi^*R\) is a principal \(Q^k\)-bundle map covering
\[
r_0\times\operatorname{id}_X:(e,x)\longmapsto(e_0,x).
\]
Therefore
\[
\widetilde r_0(\Pi^*R)\subseteq\Pi^*R\big|_{\{e_0\}\times X}\cong P.
\]

For \(g\in G\), define
\[
s(g):P\longrightarrow P
\]
as follows. Identify \(P\) with \(\Pi^*R\big|_{\{e_0\}\times X}\) and set \(s(g)(p):=\widetilde r_0(g\cdot p).\) Since the canonical \(G\)-action on \(\Pi^*R\) commutes with the principal right \(Q^k\)-action, and since \(\widetilde r_0\) is a principal \(Q^k\)-bundle map, the map \(s(g)\) is a principal \(Q^k\)-bundle map.

We now check the map covered by \(s(g)\). If \(p\in P\) lies over \(x\in X\), then, after identifying \(P\) with the restriction of \(\Pi^*R\) over \(\{e_0\}\times X\), the point \(g\cdot p\) lies over \((ge_0,\phi(g)(x))\in EG\times X.\) Applying \(\widetilde r_0\), which covers \(r_0\times\operatorname{id}_X\), sends the base point to \((e_0,\phi(g)(x)).\) Hence \(\pi_P(s(g)(p))=\phi(g)(\pi_P(p)).\) Thus \(s(g)\) covers the homeomorphism \(\phi(g):X\to X\).

Since \(s(g)\) is a principal \(Q^k\)-bundle map covering a homeomorphism of \(X\), it is a principal \(Q^k\)-bundle isomorphism. Hence \(s(g)\in \operatorname{Isom}(P)\) and, more precisely, \(s(g)\in\widehat{\operatorname{Isom}}_G(P).\)

Finally, the map
\[
G\times P\longrightarrow P, \qquad (g,p)\longmapsto s(g)(p),
\]
is continuous because it is the restriction of the composite of the continuous canonical \(G\)-action on \(\Pi^*R\) and the continuous bundle map \(\widetilde r_0\). Therefore, by the compact-open topology on \(\widehat{\operatorname{Isom}}_G(P)\), we obtain a continuous map \(s:G\longrightarrow \widehat{\operatorname{Isom}}_G(P).\) By construction, the projection
\[
\widehat{\operatorname{Isom}}_G(P)\longrightarrow G
\]
sends \(s(g)\) to \(g\). Hence \(s\) is a continuous section, i.e. a pseudo-lift.
\end{proof}

\begin{prop}
\label{prop:global-obstruction}
Let \(P\to X\) be a principal \(Q^k\)-bundle whose class lies in \(\mathcal Q^k_{\mathrm{Bor}}(X;G)\), and let
\[
s:G\to \widehat{\operatorname{Isom}}_G(P)
\]
be a pseudo-lift. Then the \(G\)-action on \(X\) lifts to \(P\) if and only if the factor set \(\omega_s\) is trivializable.
\end{prop}

\begin{proof}
If \(s\) is already a homomorphism, then it is a lift and $
\omega_s(g,h)=1,\;\forall g,h\in G$. Hence \(\omega_s\) is trivializable. Conversely, suppose that \(\omega_s\) is trivializable. Then there exists a continuous map
\[
a:G\to \mathcal G(P)
\]
such that the modified pseudo-lift $s^a(g):=a(g)s(g)$ has trivial factor set, i.e. $\omega_{s^a}(g,h)=1$. Therefore $s^a(g)s^a(h)=s^a(gh)$ for all \(g,h\in G\). Thus \(s^a\) is a continuous homomorphism and hence defines a lift of the \(G\)-action to \(P\).
\end{proof}

Summarizing, we obtain the following global criterion.

\begin{thm}[Borel condition and factor-set obstruction]
\label{thm:global-q-HY}
Let \(G\) be a compact Lie group acting on a connected, locally finite CW complex \(X\), and let $
Q^k\longrightarrow P\longrightarrow X$
be a principal \(Q^k\)-bundle. Then the \(G\)-action on \(X\) lifts to \(P\) if and only if the following two conditions hold:
\begin{enumerate}
\item[(1)] \([P]\in \mathcal Q^k_{\mathrm{Bor}}(X;G)\);
\item[(2)] for one, and hence by Lemma~\ref{lem:trivializability-independent} for every, pseudo-lift $s:G\to \widehat{\operatorname{Isom}}_G(P)$,
the associated factor set $\omega_s(g,h)=s(g)s(h)s(gh)^{-1}$
is trivializable by a continuous gauge-valued \(1\)-cochain.
\end{enumerate}
Equivalently,
\[
\mathcal Q^k_G(X)
=
\{[P]\in \mathcal Q^k_{\mathrm{Bor}}(X;G):\omega_s \text{ is trivializable}\}.
\]
\end{thm}

We now illustrate the preceding obstruction theory through examples over \(\mathbb H P^1\) and related quoric manifolds; for the latter, see \cite{ho}. We begin with the quaternionic Hopf fibration, corresponding to Chern class \(1\), where the obstruction vanishes and the standard \(Q\)-action lifts. We then show that the equivariant extension of the degree-four class \(c_2\) is not sufficient for liftability, already for the bundle of Chern class \(2\) over \(\mathbb H P^1\). Finally, we return to the general clutching construction over \(\mathbb H P^1\cong S^4\), where pseudo-lifts and their factor sets can be written explicitly.

\begin{example}[The quaternionic Hopf bundle]
\label{ex:quaternionic-hopf-lift}
Identify \(\mathbb H P^1\cong S^4.\) Let \(Q\) act on \(\mathbb H P^1\) by
\[
q\cdot[h_0:h_1]=[qh_0:h_1].
\]
Consider the quaternionic Hopf fibration
\[
Q\longrightarrow S^7\longrightarrow \mathbb H P^1.
\]
Here \(S^7=\{(h_0,h_1)\in\mathbb H^2:|h_0|^2+|h_1|^2=1\}\), with principal right \(Q\)-action \((h_0,h_1)\cdot r=(h_0r,h_1r)\). The action of \(Q\) on \(\mathbb H P^1\) lifts to \(S^7\) by \(q\cdot(h_0,h_1)=(qh_0,h_1)\). Indeed, \([qh_0:h_1]=q\cdot[h_0:h_1]\), and the lifted left \(Q\)-action commutes with the principal right \(Q\)-action:
\[
q\cdot\bigl((h_0,h_1)\cdot r\bigr)=(qh_0r,h_1r) = (qh_0,h_1)\cdot r.
\]
For later use, we recall the standard clutching coordinates on \(\mathbb H P^1\). Write
\[
\mathbb H P^1=D^4_+\cup_{S^3}D^4_-,
\]
where \(D^4_+\) has affine coordinate \(z=h_0h_1^{-1}\) and \(D^4_-\) has coordinate \(w=h_1h_0^{-1}=z^{-1}\) on the overlap. The action \(q\cdot[h_0:h_1]=[qh_0:h_1]\) is given in these coordinates by
\[
z\longmapsto qz,
\qquad
w\longmapsto wq^{-1}.
\]
For the Hopf bundle, the clutching function is \(\lambda(z)=z\). 

We now check the two conditions of Theorem~\ref{thm:global-q-HY} in this case. Here the acting group and the structural group are both \(Q\). Define
\[
\widetilde\phi:Q\longrightarrow \widehat{\operatorname{Isom}}_Q(S^7)
\]
by \(\widetilde\phi(q)(h_0,h_1)=(qh_0,h_1)\). This is a continuous homomorphism, since
\[
\widetilde\phi(q)\widetilde\phi(q')(h_0,h_1) = (qq'h_0,h_1) = \widetilde\phi(qq')(h_0,h_1).
\]
Moreover, \(\widetilde\phi(q)\) covers the given action on \(\mathbb H P^1\), because
\[
\pi(\widetilde\phi(q)(h_0,h_1))=[qh_0:h_1]=q\cdot[h_0:h_1],
\]
and it commutes with the principal right \(Q\)-action:
\[
\widetilde\phi(q)\bigl((h_0,h_1)\cdot r\bigr)=(qh_0r,h_1r)=\widetilde\phi(q)(h_0,h_1)\cdot r.
\]
Thus \(\widetilde\phi\) is a lift of the \(Q\)-action.

Consequently, Proposition~\ref{prop:lifting-implies-borel} gives
\[
[S^7]\in \mathcal Q^1_{\mathrm{Bor}}(\mathbb H P^1;Q).
\]
Equivalently, the Borel bundle
\((S^7)_Q\to(\mathbb H P^1)_Q\) restricts along the fiber inclusion to the original Hopf bundle. The second condition of Theorem~\ref{thm:global-q-HY} is also immediate: since \(\widetilde\phi\) is already a homomorphism, its factor set is identically trivial,
\[
\omega_{\widetilde\phi}(q,q')=\widetilde\phi(q)\widetilde\phi(q')\widetilde\phi(qq')^{-1}=1.
\]
In other words, both conditions of Theorem~\ref{thm:global-q-HY} are satisfied. Thus the quaternionic Hopf bundle gives the basic nontrivial example of a principal \(Q\)-bundle for which the global lifting obstruction vanishes.
\end{example}

\begin{example}[Equivariant \(c_2\)-extension is not sufficient]
\label{ex:c2-not-sufficient}
Keep the notation and \(Q\)-action on \(X=\mathbb H P^1\cong S^4\) from
Example~\ref{ex:quaternionic-hopf-lift}. Let \(Q\longrightarrow P_2\longrightarrow X\) be the principal \(Q\)-bundle with \(c_2(P_2)=2u\), where \(u\) is the generator of \(H^4(X;\mathbb Z)\). Equivalently, \(P_2\) is obtained by clutching along the equator \(S^3\subset S^4\) by the map
\[
\lambda:S^3\longrightarrow Q, \qquad \lambda(z)=z^2.
\]

The Borel construction \(X_Q=EQ\times_Q X\) is the quaternionic projectivization \(\mathbb H P(L\oplus\underline{\mathbb H})\) over \(BQ\), where \(L=EQ\times_Q\mathbb H\) is the universal quaternionic line bundle. Hence \(H_Q^4(X;\mathbb Z)\) contains a class \(U\) whose restriction to the fiber \(X\cong\mathbb H P^1\) is \(u\). Therefore \(c_2(P_2)=2u\) admits an equivariant cohomological extension, namely \(2U\). Thus the necessary condition of Corollary~\ref{cor:c2-obstruction-global} is satisfied.

We now show that the \(Q\)-action does not lift to \(P_2\). Suppose, for contradiction, that a lift exists. Recall from Example~\ref{ex:quaternionic-hopf-lift} that 
\[
\mathbb H P^1=D^4_+\cup_{S^3}D^4_-,
\]
where \(D^4_+\) has coordinate \(z\) and \(D^4_-\) has coordinate \(w=z^{-1}\) on the overlap. The action is given in these coordinates by \(z\mapsto qz\) on \(D^4_+\) and \(w\mapsto wq^{-1}\) on \(D^4_-\).

Let \(\widetilde\Phi:Q\times P_2\longrightarrow P_2\) be the assumed lifted action. Over the two trivializing charts \(P_2|_{D^4_\pm}\cong D^4_\pm\times Q\), write
\[
\widetilde\Phi_q^+ : D^4_+\times Q\longrightarrow D^4_+\times Q
\]
and
\[
\widetilde\Phi_q^- : D^4_-\times Q\longrightarrow D^4_-\times Q
\]
for the corresponding local expressions of the bundle automorphism \(\widetilde\Phi_q\).

The contraction \(z\mapsto tz\), \(t\in[0,1]\), is \(Q\)-equivariant on \(D^4_+\), since the \(Q\)-action is linear. Applying the covering homotopy theorem to this contraction, the lifted action over \(D^4_+\) is \(Q\)-equivariantly gauge-equivalent to the pullback of its restriction to the fiber over the fixed point. Equivalently, after changing the trivialization of \(P_2|_{D^4_+}\) by a \(Q\)-equivariant gauge transformation, the local expression of the lift may be assumed to be constant in the base variable. The same argument applies on \(D^4_-\). Thus, after such gauge changes on the two charts, there exist continuous homomorphisms
\[
\alpha_0,\alpha_\infty:Q\longrightarrow Q
\]
such that
\[
\widetilde\Phi_q^+(z,h)=(qz,\alpha_0(q)h), \qquad \widetilde\Phi_q^-(w,h)=(wq^{-1},\alpha_\infty(q)h).
\]

Compatibility with the clutching relation
\[
(z,h)_+\sim(z^{-1},z^2h)_-
\]
along the equator forces
\[
(qz)^2\alpha_0(q)=\alpha_\infty(q)z^2
\]
for all \(q,z\in Q\). Taking \(q=\mathbf j\), and writing \(A=\alpha_0(\mathbf j)\) and \(B=\alpha_\infty(\mathbf j)\), this condition becomes
\[
(\mathbf j z)^2A=Bz^2.
\]
Setting \(z=1\) gives \(B=-A\). Setting \(z=\mathbf i\) gives
\[
(\mathbf j\mathbf i)^2A=B\mathbf i^2,
\]
 hence \(A=B\). Thus \(A=-A\), which is impossible in \(Q=\operatorname{Sp}(1)\). Therefore the standard \(Q\)-action on \(\mathbb H P^1\) does not lift to \(P_2\).

\noindent In other words, \(c_2(P_2)\) admits an equivariant cohomological extension, but the action itself does not lift.
\end{example}

\begin{remark}[The quoric manifold example over the square]
\label{rem:square-quoric-c2-not-sufficient}
The same phenomenon occurs for the basic quoric manifold over the square \(I^2\) (see \cite[Example 4.3.4(2)]{ho}), namely
\[
M=\mathbb H P^1\times \mathbb H P^1,
\]
with its standard \(Q^2\)-action
\[
(q_1,q_2)\cdot([h_0:h_1],[g_0:g_1])
=
([q_1h_0:h_1],[q_2g_0:g_1]).
\]
Indeed, let
\[
P=\operatorname{pr}_1^*P_2\longrightarrow
\mathbb H P^1\times \mathbb H P^1,
\]
where \(P_2\to \mathbb H P^1\) is the principal \(Q\)-bundle of
Example~\ref{ex:c2-not-sufficient}. Then
\[
c_2(P)=2\,\operatorname{pr}_1^*u.
\]
Since \(c_2(P_2)\) admits an equivariant cohomological extension for the standard \(Q\)-action on the first factor, the class \(c_2(P)\) admits an equivariant cohomological extension for the product \(Q^2\)-action.

However, the \(Q^2\)-action on \(\mathbb H P^1\times \mathbb H P^1 \) does not lift to \(P\). If such a lift existed, then restricting it to the \(Q^2\)-invariant submanifold
\[
\mathbb H P^1\times\{[0:1]\}\cong \mathbb H P^1
\]
and to the subgroup \(Q\times\{1\}\subset Q^2\) would give a lift of the standard \(Q\)-action on \(\mathbb H P^1\) to \(P_2\). This contradicts Example~\ref{ex:c2-not-sufficient}.

Thus the failure of equivariant \(c_2\)-extension to characterize liftability already appears in the basic quoric manifold associated with the square.
\end{remark}

\begin{example}[The clutching calculation over \(\mathbb H P^1\)]
\label{ex:clutching-HP1}
We now return to the same clutching coordinates as in Example~\ref{ex:quaternionic-hopf-lift}, but allow arbitrary Chern class \(k\). Let \(P_k\to S^4\cong\mathbb H P^1\) be the principal \(Q\)-bundle obtained by clutching along the equator \(S^3=\{|z|=1\}\) via
\[
(z,h)_+\sim(z^{-1},z^k h)_-,
\qquad z\in S^3.
\]

For a fixed \(q\in Q\), consider on the upper chart the bundle map
\(\widetilde\Phi_q^+(z,h)=(qz,h)\). In order for this to define a bundle automorphism of \(P_k\), the lower chart expression must have the form
\[
\widetilde\Phi_q^-(w,h)=(wq^{-1},\mu_q(w)h),
\]
where the boundary value of \(\mu_q\) is forced by the clutching relation. A direct calculation gives
\[
\mu_q(w)=(qw^{-1})^k w^k,
\qquad w\in S^3.
\]

We first check that this boundary map extends over \(D^4_-\). Write
\[
\mu_q|_{S^3}=A_q\cdot B,
\qquad
A_q(w)=(qw^{-1})^k,
\qquad
B(w)=w^k.
\]
The group structure on \(Q\) induces the usual group structure on \([S^3,Q]\), and under this structure, pointwise multiplication of maps corresponds to addition of homotopy classes. Moreover, \(w\mapsto w^k\) represents \(k\in\pi_3(Q)\cong\mathbb Z\), the inversion map \(w\mapsto w^{-1}\) represents \(-1\), and left translation by a fixed element of \(Q\) is homotopic to the identity. Hence
\[
[\mu_q|_{S^3}]=[A_q]+[B]=-k+k=0\quad\text{in }\pi_3(Q).
\]
Thus \(\mu_q|_{S^3}\) is null-homotopic and extends continuously over \(D^4_-\). Consequently every element \(q\in Q\), and in particular every element of a finite subgroup \(G\subset Q\), is covered by at least one principal \(Q\)-bundle automorphism of \(P_k\).

We now choose these extensions coherently enough to compute the resulting factor set. Since \(G\) is finite, any set-theoretic choice of lifts gives a continuous pseudo-lift. For each \(q\in G\), choose a path \(p^q_r\) in \(Q\) from \(1\) to \(q\), and define an extension by
\[
\mu_q(ru)=\bigl(p^q_r u^{-1}\bigr)^k u^k,
\qquad
r\in[0,1],
\quad u\in S^3.
\]
This has boundary value
\[
\mu_q(u)=(qu^{-1})^k u^k
\]
when \(r=1\). Hence it defines a lift \(\widetilde\Phi_q\) of the element \(q\). Choosing these lifts for all \(q\in G\) gives a pseudo-lift
\[
s:G\longrightarrow \widehat{\operatorname{Isom}}_G(P_k),
\qquad
s(q)=\widetilde\Phi_q.
\]

The factor set
\[
\omega_s(q,q')=s(q)s(q')s(qq')^{-1}
\]
is a gauge transformation of \(P_k\). In the upper trivialization it is the identity. In the lower trivialization it is represented by the map
\[
\nu_{q,q'}(w)
=
\mu_q(wq'^{-1})\,\mu_{q'}(w)\,\mu_{qq'}(w)^{-1},
\qquad
w\in D^4_-.
\]
On the boundary \(S^3\), this map is identically \(1\). Indeed, using the boundary formula for \(\mu_q\), the factors \((wq'^{-1})^k\) and \((q'w^{-1})^k\) cancel because \(wq'^{-1}\) and \(q'w^{-1}\) are inverse elements of \(Q\). Therefore \(\nu_{q,q'}\) glues with the constant map \(1\) on the upper chart to give a based map
\[
S^4\longrightarrow Q.
\]
We denote its homotopy class by
\[
[\nu_{q,q'}]\in\pi_4(Q)\cong\mathbb Z/2.
\]

This class records the path component of the gauge transformation \(\omega_s(q,q')\). More precisely, the evaluation fibration for the gauge group of \(P_k\) gives, since \(\pi_1(Q)=\pi_0(Q)=0\), an isomorphism 
\[\pi_0(\mathcal G(P_k))\cong \pi_4(Q)\cong\mathbb Z/2.
\]
Under this isomorphism, the component of \(\omega_s(q,q')\) is represented by \([\nu_{q,q'}]\). Thus, after passing to path components of the gauge group, the factor set determines an ordinary \(2\)-cocycle
\[
\overline{\omega_s}:G\times G\longrightarrow \mathbb Z/2, \qquad \overline{\omega_s}(q,q')=[\nu_{q,q'}].
\]
Changing the pseudo-lift changes this cocycle by a coboundary. Since \(\operatorname{Aut}(\mathbb Z/2)\) is trivial, this gives a well-defined cohomology class
\[
[\overline{\omega_s}]\in H^2(G;\mathbb Z/2).
\]
If this class is nonzero, then the original gauge-valued factor set \(\omega_s\) cannot be trivialized.

For \(k=1\), this component-level obstruction vanishes. Indeed, with the preceding choice $\mu_q(ru)=(p^q_r u^{-1})u$,
a direct computation gives
\[
\nu_{q,q'}(ru)=p^q_r\,p^{q'}_r\,(p^{qq'}_r)^{-1}.
\]
This expression depends only on \(r\), not on \(u\). Hence the corresponding map \(S^4\to Q\) factors through \(S^1\). Since \(\pi_1(Q)=0\), it is null-homotopic. Therefore
\[
[\overline{\omega_s}]=0
\quad\text{for }k=1,
\]
in agreement with the explicit lift of the quaternionic Hopf bundle from Example~\ref{ex:quaternionic-hopf-lift}.
\end{example}

\begin{remark}[The case \(k\geq 2\)]
\label{rem:clutching-k2}
For \(k\geq 2\), the argument used for \(k=1\) no longer applies. The terms appearing in the lower-chart representative
\[
\nu_{q,q'}(w) =\mu_q(wq'^{-1})\,\mu_{q'}(w)\,\mu_{qq'}(w)^{-1}
\]
involve \(k\)-fold products of noncommuting quaternions. Thus the elementary cancellation which occurs for \(k=1\) does not persist for higher \(k\). Equivalently, one cannot simplify expressions of the form \((xy)^k\) to \(x^ky^k\) unless \(x\) and \(y\) commute.

This failure is not merely notational. For instance, the simpler candidate formula
\[
\widetilde\Phi_q^-(w,h)=(wq^{-1},q^k h)
\]
does not define a lift for \(k\geq 2\). Indeed, compatibility with the clutching relation would force
\[
(qz)^k=q^kz^k
\]
for all \(z\in S^3\), which is false in \(Q=\operatorname{Sp}(1)\) unless \(q\) commutes with \(z\).

A more conceptual way to view the problem is the following. A lift of a finite subgroup \(G\subset Q\) to \(P_k\) would require a representative
\[
\beta_k:S^3\longrightarrow Q
\]
of degree \(k\) for the clutching class, together with isotropy homomorphisms
\[
\rho_+,\rho_-:G\longrightarrow Q
\]
at the two poles, such that
\[
\beta_k(qz)=\rho_-(q)\,\beta_k(z)\,\rho_+(q)^{-1}
\]
for every \(q\in G\) and \(z\in S^3\). Thus the lifting problem can be viewed as an equivariant representative problem for the degree-\(k\) clutching map.

In special cases, this equivariance condition can be related to self-maps of spherical space forms. For example, if \(\rho_+\) is trivial and \(\rho_-\) is the inclusion \(G\hookrightarrow Q\), then such a \(\beta_k\) would descend to a map on the quotient \(S^3/G\). Hence the possible degrees are constrained by the homotopy theory of the spherical space form \(S^3/G\), equivalently by the induced maps on the cohomology of \(BG\). For \(G=Q_8\), this suggests that the vanishing of the component obstruction
\[
[\overline{\omega_s}]\in H^2(G;\mathbb Z/2)
\]
should be related to a finite calculation involving \(BQ_8\).
\end{remark}

The examples above are the first instances of a more general construction in toric topology. Principal quaternionic torus bundles arise naturally as quotients of quaternionic moment-angle manifolds by freely acting quaternionic subtori; see \cite{ho, gken} and the references therein.

\begin{remark}
\label{rem:moment-angle-quoric-examples}
Quaternionic moment-angle manifolds also provide a natural source of examples. Let \(P^n\) be a simple polytope with \(m\) facets, and let \(\mathcal Z_P^{\mathbb H}\) be the associated quaternionic moment-angle manifold. The coordinatewise action of \(Q^m=(\operatorname{Sp}(1))^m\) on \(\mathcal Z_P^{\mathbb H}\) contains, whenever a global characteristic map exists, a freely acting subgroup \(K\cong Q^{m-n}.\) The quotient is the corresponding quoric manifold \(\mathcal Z_P^{\mathbb H}/K\cong M(P,\lambda).\) Equivalently, one has a principal quaternionic torus bundle
\[
Q^{m-n}
\longrightarrow
\mathcal Z_P^{\mathbb H}
\longrightarrow
M(P,\lambda).
\]

The first example is obtained when \(P=[0,1]\). In this case \(n=1\) and \(m=2\), so \(m-n=1\). The quaternionic moment-angle manifold is
\[
\mathcal Z_{[0,1]}^{\mathbb H}\cong S^7,
\]
while the corresponding quoric manifold is
\[
M([0,1],\lambda)\cong\mathbb H P^1\cong S^4.
\]
Thus the above principal bundle becomes the quaternionic Hopf fibration
\[
Q\longrightarrow S^7\longrightarrow \mathbb H P^1.
\]
In this sense, the quaternionic Hopf fibration is the basic example of the moment-angle--to--quoric quotient construction.
\end{remark}

The examples above show how the factor-set obstruction can be computed in concrete quaternionic situations. They also indicate why the nonabelian case does not reduce formally to the classical torus case: the Borel-image condition, the existence of pseudo-lifts, and the trivializability of the factor set are logically distinct pieces of data. We close this subsection with two broader comments which place these phenomena in a more conceptual homotopy-theoretic framework.

\begin{remark}[A moduli-theoretic reformulation]
\label{rem:moduli-future-work}
The obstruction theory developed in this section is formulated in terms of isomorphism classes of principal bundles, pseudo-lifts, and gauge-valued factor sets, following the classical approach of Hattori--Yoshida \cite{hatyos}. A natural alternative, closer in spirit to the equivariant bundle theory of Lashof--May--Segal~\cite{lms}, would replace the discrete set \(\mathcal Q^k(X)\) by the mapping space \(\operatorname{Map}(X,BQ^k)\) and reformulate the lifting problem as a comparison between Borel-type equivariant mapping spaces, homotopy fixed points, and genuinely \(G\)-equivariant principal bundle data.

From this perspective, the Borel-image condition records homotopy-theoretic equivariant data, while the existence of a pseudo-lift and the trivializability of its factor set encode the additional conditions needed to obtain a lifted action on the total space of the bundle. Such a reformulation would package the coherence data more conceptually and would likely extend beyond the compact Lie group setting considered here.

We expect this perspective to recover the results of the present section and to clarify the relation between Borel-image data, pseudo-lifts, and the gauge-valued obstruction appearing in Example~\ref{ex:clutching-HP1}. We leave this direction as future work.
\end{remark}

The same issue can also be viewed from the perspective of gauge groups. Indeed, the factor set takes values in the gauge group of the bundle, and the topology of this gauge group already contains subtle unstable homotopy-theoretic information in the basic quaternionic case.

\begin{remark}[Gauge groups and stabilization]
\label{rem:gauge-groups-stabilization}
The gauge groups appearing in the quaternionic lifting problem are already highly nontrivial in the basic case of principal \(Q=\operatorname{Sp}(1)\) bundles over \(S^4\). If \(P_k\to S^4\) is the principal \(G\)-bundle classified by \(k\in [S^4,BG]\cong\mathbb Z\), then the classifying space of its gauge group is homotopy equivalent to the component of the mapping space which contains the classifying map of \(P_k\): \(B\mathcal G_k(G)\simeq \operatorname{Map}_k(S^4,BG)\). Moreover, the evaluation fibration gives \(\operatorname{Map}^*_k(S^4,BG) \longrightarrow\operatorname{Map}_k(S^4,BG)\xrightarrow{\mathrm{ev}} BG, \) and \(\operatorname{Map}^*_k(S^4,BG)\simeq \Omega^3_0G.\) The corresponding connecting map \(G\longrightarrow \Omega^3_0G\) is governed, after adjunction, by a Samelson product \(S^3\wedge G\longrightarrow G\). Thus the nonabelian nature of the gauge group is closely related to unstable Samelson products. This viewpoint is used systematically in the study of gauge groups over \(S^4\); see, for example, Kishimoto--Theriault--Tsutaya \cite{KishimotoTheriaultTsutayaG2}.

For example, in the case \(G=SU(2)\cong \operatorname{Sp}(1)\), the homotopy type of the gauge group of the bundle \(P_k\to S^4\) depends on the integer \(k\) through the classical condition (see \cite{Kono91})
\[
\mathcal G_k(SU(2))\simeq \mathcal G_{k'}(SU(2))
\quad\Longleftrightarrow\quad
(12,k)=(12,k').
\]
This illustrates that the coefficient group in the factor-set obstruction is not homotopically uniform as the quaternionic bundle varies.

There is also a stable comparison. The stabilization \(\operatorname{Sp}(1)\longrightarrow \operatorname{Sp}(\infty) \) passes from the unstable quaternionic gauge group to a stable infinite-loop-space setting. In that stable setting the relevant mapping spaces are homotopy-commutative, so the obstruction theory becomes closer in spirit to the abelian theory of torus bundles. The passage back from the stable theory to the unstable \(Q=\operatorname{Sp}(1)\) theory is precisely where unstable Samelson products and the \(k\)-dependence of the gauge group enter. We do not pursue this stable comparison further here.
\end{remark}

\subsection{Enumeration of lifts}

We now record the quaternionic analogue of the enumeration result of Hattori--Yoshida \cite[\S 4]{hatyos}. Assume that the \(G\)-action on \(X\) admits at least one lift to the principal \(Q^k\)-bundle $Q^k\longrightarrow P\longrightarrow X$.
Fix such a lift to be $\widetilde\phi:G\longrightarrow \widehat{\operatorname{Isom}}_G(P)$.
Then every other lift \(\widetilde\phi'\) differs from \(\widetilde\phi\) by a continuous gauge-valued map $a:G\longrightarrow \mathcal G(P)$,
namely $\widetilde\phi'(g)=a(g)\widetilde\phi(g)$.
The condition that \(\widetilde\phi'\) is a homomorphism is equivalent to
\[
a(gh)=a(g)\,{}^{\widetilde\phi(g)}a(h),
\]
where
\[
{}^{\widetilde\phi(g)}\gamma
=
\widetilde\phi(g)\gamma\widetilde\phi(g)^{-1}
\]
for \(\gamma\in\mathcal G(P)\). Thus \(a\) is a nonabelian crossed homomorphism. 

\begin{definition}
\label{def:gauge-equivalence-liftings}
Let
\[
\widetilde\phi,\widetilde\phi':
G\longrightarrow \widehat{\operatorname{Isom}}_G(P)
\]
be two lifts to $P$ of the \(G\)-action on \(X\). We say that \(\widetilde\phi\) and \(\widetilde\phi'\) are \emph{gauge equivalent} if there exists a gauge transformation \(b\in\mathcal G(P)\) such that
\[\widetilde\phi'(g) = b\,\widetilde\phi(g)\,b^{-1}
,\;\;\forall g\in G.\]
\end{definition}

\noindent With this terminology, the relevant equivalence relation on crossed homomorphisms is the following. If \(b\in\mathcal G(P)\), then conjugating a lift by \(b\) changes the corresponding crossed homomorphism \(a\) to
\[
a^b(g)
=
b\,a(g)\,{}^{\widetilde\phi(g)}b^{-1}.
\]
The quotient of the set of continuous crossed homomorphisms, i.e. maps $a:G\longrightarrow \mathcal G(P)$ such that $a(gh)=a(g)\,{}^{\widetilde\phi(g)}a(h)$, by this equivalence relation is denoted by $H^1_{\mathrm{nab}}\bigl(G;\mathcal G(P)_{\widetilde\phi}\bigr)$. The distinguished point is the class of the trivial crossed homomorphism \(a(g)=1\), corresponding to the fixed lift \(\widetilde\phi\).

\begin{remark}
\label{rem:nonabelian-H1-pointed-set}
The notation
\[
H^1_{\mathrm{nab}}
\bigl(G;\mathcal G(P)_{\widetilde\phi}\bigr)
\]
is used abusively and by analogy with group cohomology. However, no cochain complex or differential is involved here. The object is a pointed set of equivalence classes of crossed homomorphisms, not a cohomology group in the usual sense. We retain the notation because it is standard in nonabelian cohomology, but we emphasize throughout that it has a set-theoretic, rather than group-theoretic, nature.
\end{remark}

\begin{thm}[Enumeration of quaternionic lifts]
\label{thm:enumeration-global-liftings}
Let \(G\) be a compact Lie group acting on \(X\), and let $Q^k\longrightarrow P\longrightarrow X$
be a principal \(Q^k\)-bundle. Suppose that the \(G\)-action admits at least one lift $\widetilde\phi:G\longrightarrow \widehat{\operatorname{Isom}}_G(P)$.
Then the set of lifts of the \(G\)-action to \(P\), modulo conjugation by gauge transformations, is naturally identified with the pointed set $H^1_{\mathrm{nab}}(G;\mathcal G(P)_{\widetilde\phi})$.
\end{thm}
\begin{proof}
Fix a lift $\widetilde\phi:G\longrightarrow \widehat{\operatorname{Isom}}_G(P)$.
Let $\widetilde\phi':G\longrightarrow \widehat{\operatorname{Isom}}_G(P)$
be any other lift. Since \(\widetilde\phi(g)\) and \(\widetilde\phi'(g)\) both cover the same homeomorphism \(\phi(g)\) of \(X\), there is a unique gauge transformation $a(g)\in\mathcal G(P)$
such that $\widetilde\phi'(g)=a(g)\widetilde\phi(g)$.
Continuity of \(a\) follows from the continuity of \(\widetilde\phi\) and \(\widetilde\phi'\).

We now compute the condition that \(\widetilde\phi'\) is a homomorphism. For \(g,h\in G\), we have
\[
\widetilde\phi'(g)\widetilde\phi'(h)
=
a(g)\widetilde\phi(g)a(h)\widetilde\phi(h)
=
a(g)\,{}^{\widetilde\phi(g)}a(h)\,\widetilde\phi(gh),
\]
where
\[
{}^{\widetilde\phi(g)}a(h)
=
\widetilde\phi(g)a(h)\widetilde\phi(g)^{-1}.
\]
On the other hand, \(\widetilde\phi'(gh)=a(gh)\widetilde\phi(gh).\) Thus \(\widetilde\phi'\) is a lift if and only if \(a(gh)=a(g)\,{}^{\widetilde\phi(g)}a(h).\) Hence lifts of the \(G\)-action to \(P\) correspond bijectively to continuous crossed homomorphisms $a:G\longrightarrow\mathcal G(P)$.

It remains to identify gauge equivalence. Let \(b\in\mathcal G(P)\). The gauge conjugate of \(\widetilde\phi'\) by \(b\) is
\[
g\longmapsto b\,\widetilde\phi'(g)\,b^{-1}.
\]
Using \(\widetilde\phi'(g)=a(g)\widetilde\phi(g),\) we obtain \[b\,\widetilde\phi'(g)\,b^{-1} =b\,a(g)\widetilde\phi(g)b^{-1}=b\,a(g)\,{}^{\widetilde\phi(g)}b^{-1}\,\widetilde\phi(g).\] Therefore the crossed homomorphism corresponding to the conjugate lift is \(a^b(g)=b\,a(g)\,{}^{\widetilde\phi(g)}b^{-1},\) which is exactly the equivalence relation defining the pointed set $H^1_{\mathrm{nab}} \bigl(G;\mathcal G(P)_{\widetilde\phi}\bigr)$.
\end{proof}

\begin{cor}[Uniqueness under vanishing nonabelian $H^1$]
With the hypotheses of Theorem~\ref{thm:enumeration-global-liftings}, if
\[
H^1_{\mathrm{nab}}(G;\mathcal G(P)_{\widetilde\phi})
\]
is trivial, then the lifting of the $G$-action to $P$ is unique up to gauge conjugacy.
\end{cor}

\iffalse

{\color{blue} \begin{example}[Enumeration for the quaternionic Hopf bundle]
\label{ex:enumeration-hopf}
Consider again the quaternionic Hopf fibration
\[
Q\longrightarrow S^7\longrightarrow \mathbb H P^1
\]
from Example~\ref{ex:quaternionic-hopf-lift}, together with the action of \(Q\) on \(\mathbb H P^1\) given by \(q\cdot[h_0:h_1]=[qh_0:h_1]\). This action admits the lift \(\widetilde\phi(q)(h_0,h_1)=(qh_0,h_1)\). Therefore Theorem~\ref{thm:enumeration-global-liftings} applies.

In this case, the set of all lifts of the \(Q\)-action to the Hopf bundle, modulo gauge conjugacy, is identified with the pointed nonabelian set \(H^1_{\mathrm{nab}}
\bigl(
Q;\mathcal G(S^7)_{\widetilde\phi}
\bigr),
\) where \(\mathcal G(S^7)
\cong
\Gamma(\operatorname{Ad}(S^7))
\) is the gauge group of the quaternionic Hopf bundle and \(Q\) acts on it by conjugation through the fixed lift \(\widetilde\phi\).

Thus, even in the basic Hopf case, the enumeration of lifts is controlled by the nonabelian gauge group of the bundle. The distinguished point of this pointed set corresponds to the standard lift \(\widetilde\phi(q)(h_0,h_1)=(qh_0,h_1)\). \textcolor{red}{what is the use of this example, we just rewrite the theory for this fibration, right? I think Theriault meant an actual computation of the nonabelian set.}
\end{example}}

\fi

\begin{example}[Enumeration over a point]
\label{ex:enumeration-point}
Let \(X=\{\ast\}\), and let \(G=Q\) act trivially on \(X\). The unique principal \(Q\)-bundle over \(X\) is
\[
Q\longrightarrow Q\longrightarrow \{\ast\},
\]
where the principal right action is right multiplication.

A lift of the trivial action on \(X\) is the same thing as a continuous action of \(Q\) on the total space \(Q\) by principal bundle automorphisms. Every principal bundle automorphism of \(Q\to\{\ast\}\) is given by left multiplication by an element of \(Q\). Hence a lift is equivalently a continuous homomorphism
\[
\alpha:Q\longrightarrow Q,
\]
acting by
\[
q\cdot h=\alpha(q)h.
\]

Two such lifts are gauge equivalent precisely when the corresponding homomorphisms are conjugate in \(Q\). Therefore the set of lifts modulo gauge conjugacy is
\[
\operatorname{Hom}_{\mathrm{cts}}(Q,Q)/Q,
\]
where \(Q\) acts by conjugation on the target.

If we choose the trivial lift as the distinguished lift, then the induced action of \(Q\) on the gauge group \(Q\) is trivial. Thus the relevant pointed nonabelian set is
\[
H^1_{\mathrm{nab}}(Q;Q_{\mathrm{triv}}).
\]
Since \(Q=\operatorname{Sp}(1)\) is simple, every continuous homomorphism \(Q\to Q\) is either trivial or an automorphism. Moreover, every automorphism of is inner; see \cite[p.25]{ho}. Hence there are exactly two gauge equivalence classes of lifts: the trivial lift and the standard left multiplication lift. Therefore
\[
H^1_{\mathrm{nab}}(Q;Q_{\mathrm{triv}})
\cong
\{\text{trivial class},\text{standard class}\}
\]
as a pointed set.
\end{example}

\section{Lifts of untwisted local quaternionic torus actions}
\label{sec:untwisting}

\subsection{Quaternionic torus actions}
\label{subsec:orbit-space-globalization}

In this section, we recall from \cite{bg25} the orbit space of a local quaternionic toric action and the obstruction to the existence of a globally defined quaternionic torus action; proofs are omitted since they appear in detail in \cite{bg25}.

The group \(Q^n\) acts on \(\mathbb H^n\) by coordinatewise quaternionic multiplication:
\[
(q_1,\ldots,q_n)\cdot(h_1,\ldots,h_n)
=
(q_1h_1,\ldots,q_nh_n).
\]
This action is called the \emph{regular \(Q^n\)-action on \(\mathbb H^n\)}. The orbit space of the regular action is naturally identified with the standard positive cone
\[
\mathbb R_{\geq 0}^n
=
\lbrace (\xi_1,\ldots,\xi_n)\in\mathbb R^n: \xi_i\geq 0,\ i=1,\ldots,n
\rbrace.
\]
Indeed, the map $\mu_{\mathbb H^n}:\mathbb H^n\longrightarrow\mathbb R_{\geq 0}^n,\;\;
\mu_{\mathbb H^n}(h_1,\ldots,h_n)
=\bigl(|h_1|^2,\ldots,|h_n|^2\bigr)$,
is invariant under the regular \(Q^n\)-action and induces a homeomorphism $\mathbb H^n/Q^n\longrightarrow\mathbb R_{\geq 0}^n$.
Both spaces carry natural stratifications. The stratification of \(\mathbb R_{\geq 0}^n\) is determined by the number of coordinates which vanish, and under the above homeomorphism it corresponds to the orbit-type stratification of the regular \(Q^n\)-action.

We recall that a locally regular \(Q^n\)-action on a \(4n\)-dimensional manifold \(M\) is a globally defined \(Q^n\)-action such that every point of \(M\) has a \(Q^n\)-invariant neighborhood which, up to an automorphism of \(Q^n\), is equivariantly diffeomorphic to a \(Q^n\)-invariant open subset of \(\mathbb H^n\).

More generally, a local \(Q^n\)-action on \(M\) is represented by a weakly regular atlas
\[
\lbrace
(U_\alpha^M,\varphi_\alpha^M)
\rbrace_{\alpha\in A},
\]
where each $\varphi_\alpha^M:U_\alpha^M\longrightarrow\mathbb H^n$
is a homeomorphism onto a \(Q^n\)-invariant open subset and, for every nonempty overlap $U_{\alpha\beta}^M
=U_\alpha^M\cap U_\beta^M$,
there exists an automorphism $\rho_{\alpha\beta}\in\operatorname{Aut}(Q^n)$
such that the overlap map $\varphi_{\alpha\beta}^M
:=
\varphi_\alpha^M\circ(\varphi_\beta^M)^{-1}$ is \(\rho_{\alpha\beta}\)-equivariant.

\noindent Once a local \(Q^n\)-action is given, we will always use the maximal weakly regular atlas representing it.

Let \((M^{4n},\mathcal Q)\) be a \(4n\)-dimensional manifold equipped with a local \(Q^n\)-action, and let $\lbrace
(U_\alpha^M,\varphi_\alpha^M)
\rbrace_{\alpha\in A}$
be its maximal weakly regular atlas. We endow each quotient space $\varphi_\alpha^M(U_\alpha^M)/Q^n$
 with the quotient topology induced by the natural projection
\[
\pi_\alpha:
\varphi_\alpha^M(U_\alpha^M)
\longrightarrow
\varphi_\alpha^M(U_\alpha^M)/Q^n.
\]
For every nonempty overlap \(U_{\alpha\beta}^M\), the map \(\varphi_{\alpha\beta}^M\) induces a homeomorphism
\[
\varphi_\beta^M(U_{\alpha\beta}^M)/Q^n
\longrightarrow
\varphi_\alpha^M(U_{\alpha\beta}^M)/Q^n.
\]

Two elements $b_\alpha\in\varphi_\alpha^M(U_\alpha^M)/Q^n$ and $b_\beta\in\varphi_\beta^M(U_\beta^M)/Q^n$
are said to be \emph{equivalent} if
\[
b_\alpha\in\varphi_\alpha^M(U_{\alpha\beta}^M)/Q^n,
\qquad
b_\beta\in\varphi_\beta^M(U_{\alpha\beta}^M)/Q^n,
\]
and the homeomorphism induced by \(\varphi_{\alpha\beta}^M\) sends
\(b_\beta\) to \(b_\alpha\). This defines an equivalence relation \(\sim_{\mathrm{orb}}\) on  the disjoint union
\[
\bigsqcup_{\alpha\in A}
\varphi_\alpha^M(U_\alpha^M)/Q^n.
\]

\begin{definition}
The orbit space of the local \(Q^n\)-action \(\mathcal Q\) is the quotient space
\[
B_M
:=
\left(
\bigsqcup_{\alpha\in A}
\varphi_\alpha^M(U_\alpha^M)/Q^n
\right)\big/\sim_{\mathrm{orb}},
\]
endowed with the quotient topology.
\end{definition}

By construction, the local maps
\[
\pi_\alpha\circ\varphi_\alpha^M:
U_\alpha^M
\longrightarrow
\varphi_\alpha^M(U_\alpha^M)/Q^n
\]
glue to an open continuous map
\[
\pi_M:M\longrightarrow B_M,
\]
called the \emph{orbit map} of the local \(Q^n\)-action.

\begin{prop}[{\cite[Proposition~2.7]{bg25}}]
\label{prop:orbit-space-corners}
The orbit space \(B_M\) carries a natural structure of an \(n\)-dimensional topological manifold with corners.
\end{prop}

\begin{remark}\label{rem:orbit-atlas}
The atlas $\lbrace
(U_\alpha^B,\varphi_\alpha^B)
\rbrace_{\alpha\in A}$
satisfies $U_\alpha^M=\pi_M^{-1}(U_\alpha^B)$
and \(\varphi_\alpha^M(U_\alpha^M) = \mu_{\mathbb H^n}^{-1} \bigl(\varphi_\alpha^B(U_\alpha^B)\bigr).\) Moreover, the following diagram commutes:
\[
\begin{tikzcd}
\pi_M^{-1}(U_\alpha^B)
\arrow[r, "\varphi_\alpha^M"]
\arrow[d, "\pi_M"']
&
\mu_{\mathbb H^n}^{-1}
\bigl(\varphi_\alpha^B(U_\alpha^B)\bigr)
\arrow[d, "\mu_{\mathbb H^n}"]
\\
U_\alpha^B
\arrow[r, "\varphi_\alpha^B"']
&
\varphi_\alpha^B(U_\alpha^B)
\end{tikzcd}
\]
\end{remark}

\noindent We now describe the obstruction to the existence of a globally defined \(Q^n\)-action. Let $\mathcal U
=\{U_\alpha^B\}_{\alpha\in A}$
be the open covering of \(B_M\) induced by the maximal weakly regular atlas. The automorphisms $\rho_{\alpha\beta}\in\operatorname{Aut}(Q^n)$
associated with the overlap maps form a Čech \(1\)-cocycle on \(\mathcal U\) with values in \(\operatorname{Aut}(Q^n)\). We denote its cohomology class by $[\rho_{\alpha\beta}]
\in
\check H^1
\bigl(B_M;\operatorname{Aut}(Q^n)\bigr)$.

The following class of examples is the quaternionic analogue of locally toric Lagrangian fibrations in the symplectic torus setting (see details in \cite[\S 5]{bg25}).

\begin{example}[Generalized Lagrangian-type quaternionic toric fibrations]
\label{ex:generalized-lagrangian-quaternionic-fibration}
Let \((M^{4n},\psi)\) be a smooth \(4n\)-dimensional manifold equipped with a tetraplectic structure, and let \(B\) be an \(n\)-dimensional manifold with corners. A continuous map \(\pi:(M^{4n},\psi)\longrightarrow B\) is called a \emph{locally generalized Lagrangian-type quaternionic toric fibration} if, for every point of \(B\), there exists a coordinate neighborhood \((U_\alpha^B,\varphi_\alpha^B)\) modeled on \(\mathbb R_{\geq 0}^n\), and a tetraplectomorphism \(\varphi_\alpha^M:\bigl(\pi^{-1}(U_\alpha^B),\psi\bigr)\longrightarrow\bigl(\mu_{\mathbb H^n}^{-1}(\varphi_\alpha^B(U_\alpha^B)),\psi_{\mathbb H^n}\bigr)\) such that \( \mu_{\mathbb H^n}\circ\varphi_\alpha^M = \varphi_\alpha^B\circ\pi\). Here \(\mu_{\mathbb H^n}:\mathbb H^n\longrightarrow\mathbb R_{\geq 0}^n\) is the orbit map of the regular \(Q^n\)-action on \(\mathbb H^n\).

Such a fibration carries a natural local \(Q^n\)-action. Indeed, the local models are given by the regular \(Q^n\)-action on \(\mathbb H^n\), and the transition maps between local trivializations are equivariant up to automorphisms of \(Q^n\). Thus locally generalized Lagrangian-type quaternionic toric fibrations provide geometric examples of the local quaternionic torus actions considered in this paper.

In particular, when a quaternionic toric manifold in the sense of Gentili--Gori--Sarfatti \cite{ggs} is equipped with a tri-moment map \(\mu:(M^{4n},\psi)\longrightarrow \mathbb R^n,\) and the local normal form is modeled on the regular \(Q^n\)-action on \(\mathbb H^n\), the tri-moment map gives an example of such a locally generalized Lagrangian-type quaternionic toric fibration.
\end{example}

\begin{prop}[{\cite[Proposition~2.12]{bg25}}]
\label{prop:globalization-obstruction}
A local \(Q^n\)-action \(\mathcal Q\) on \(M\) is induced by a locally regular \(Q^n\)-action if and only if the cocycle \({\rho_{\alpha\beta}}\) is cohomologous to the trivial Čech \(1\)-cocycle in \(\check H^1 \bigl(B_M;\operatorname{Aut}(Q^n)\bigr)\).
\end{prop}

\begin{remark}
If \(B_M\) is connected, the class
\[
[\rho_{\alpha\beta}]
\in
\check H^1
\bigl(B_M;\operatorname{Aut}(Q^n)\bigr)
\]
may equivalently be described by a monodromy representation
\[
\rho:\pi_1(B_M)\longrightarrow\operatorname{Aut}(Q^n),
\]
up to conjugation. In particular, the local action is induced by a globally defined locally regular \(Q^n\)-action precisely when the corresponding monodromy is trivial up to the usual equivalence. More precisely, there is a one-to-one correspondence between \(\check H^1 \bigl(B_M;\operatorname{Aut}(Q^n)\bigr) \) and the moduli space of representations of \(\pi_1(B_M)\) to \(Aut(Q^n)\).
\end{remark}

\subsection{Untwisting the action}
Let \((M,\mathcal Q)\) be a \(4n\)-dimensional manifold equipped with a local \(Q^n\)-action \(\mathcal Q\). Let $\lbrace
(U_\alpha^M,\varphi_\alpha^M)
\rbrace_{\alpha\in A}$
be a weakly regular atlas belonging to \(\mathcal Q\), and let $\lbrace
(U_\alpha^B,\varphi_\alpha^B)
\rbrace_{\alpha\in A}$
be the induced atlas of the orbit space \(B_M\), satisfying the properties described in Remark~\ref{rem:orbit-atlas}. Let
\[
p:\widetilde B_M\longrightarrow B_M
\]
be the universal covering of \(B_M\). We regard \(p\) as a principal \(\pi_1(B_M)\)-bundle, with \(\pi_1(B_M)\) acting on the right on \(\widetilde B_M\). Consider the fiber product
\[
\widetilde M
:=
p^*M
=
\lbrace
(\widetilde b,x)\in\widetilde B_M\times M:
p(\widetilde b)=\pi_M(x)
\rbrace,
\]
where $\pi_M:M\longrightarrow B_M$
is the orbit map. The projection
\[
\widetilde\pi_M:\widetilde M\longrightarrow\widetilde B_M,
\qquad
(\widetilde b,x)\longmapsto\widetilde b,
\]
is the pullback of \(\pi_M\). The local \(Q^n\)-action on \(M\) pulls back naturally to a local \(Q^n\)-action on \(\widetilde M\), whose orbit space is \(\widetilde B_M\). Since \(\widetilde B_M\) is simply connected, the corresponding Čech class in
 $\check H^1
\bigl(\widetilde B_M;\operatorname{Aut}(Q^n)\bigr)$
is trivial. Hence, by Proposition~\ref{prop:globalization-obstruction}, the pulled-back local action is induced by a globally defined locally regular \(Q^n\)-action.

Moreover, \(\widetilde M\) carries the natural action of \(\pi_1(B_M)\) induced by the deck transformations:
\[
a\cdot(\widetilde b,x)
=
(\widetilde b\cdot a^{-1},x),
\qquad
a\in\pi_1(B_M).
\]
The global \(Q^n\)-action and the deck-transformation action combine to give an action of a semidirect product. We now describe this action explicitly.

By replacing the covering \(\{U_\alpha^B\}_{\alpha\in A}\) by a refinement if necessary, we may assume that for every \(\alpha\) there is a local trivialization
\[
\widetilde\varphi_\alpha^B:
p^{-1}(U_\alpha^B)
\longrightarrow
U_\alpha^B\times\pi_1(B_M),
\]
of the universal covering. On every nonempty overlap $U_{\alpha\beta}^B
=
U_\alpha^B\cap U_\beta^B$,
let $a_{\alpha\beta}:U_{\alpha\beta}^B
\longrightarrow\pi_1(B_M)$
be the corresponding transition function. Since \(\pi_1(B_M)\) is discrete, the functions \(a_{\alpha\beta}\) are locally constant. We choose the trivializations so that, whenever
\[
\widetilde\varphi_\alpha^B(\widetilde b)
=
\bigl(p(\widetilde b),a_\alpha\bigr),
\qquad
\widetilde\varphi_\beta^B(\widetilde b)
=\bigl(p(\widetilde b),a_\beta\bigr),
\]
one has \(a_\alpha=a_{\alpha\beta}a_\beta\).

As explained in \cite{bg25}, recall that the automorphisms $\rho_{\alpha\beta}\in\operatorname{Aut}(Q^n)$
associated with the weakly regular atlas define a class $[{\rho_{\alpha\beta}}]
\in
\check H^1
\bigl(B_M;\operatorname{Aut}(Q^n)\bigr)$.
Choose a representative
\[
\rho:\pi_1(B_M)\longrightarrow\operatorname{Aut}(Q^n),
\]
of the corresponding conjugacy class of representations. Then, for every
\(\alpha\), there exists an automorphism $\rho_\alpha\in\operatorname{Aut}(Q^n)$
such that, on every nonempty overlap, $\rho_{\alpha\beta}
=
\rho_\alpha\circ\rho(a_{\alpha\beta})
\circ\rho_\beta^{-1}$. Let \[\mathcal G_M:=Q^n\rtimes_\rho\pi_1(B_M),\]
be the semidirect product determined by \(\rho\). As a set, it is the Cartesian product $Q^n\times\pi_1(B_M)$,
with multiplication
\[
(u_1,a_1)(u_2,a_2)
=\bigl(u_1\rho(a_1)(u_2),a_1a_2\bigr).
\]

Let $(u,a)\in \mathcal G_M$
and $(\widetilde b,x)\in\widetilde M$.
Suppose that $\widetilde b\in p^{-1}(U_\alpha^B)$
and that $
\widetilde\varphi_\alpha^B(\widetilde b)
=\bigl(p(\widetilde b),a_\alpha\bigr)$.
Consider the map 
\[\cdot_\rho:\;\;\mathcal G_M\times \widetilde M\longrightarrow \widetilde M\]
\begin{equation}
\label{eq:untwisted-action}
(u,a)\cdot_\rho(\widetilde b,x)
:=
\left(
\widetilde b\cdot a^{-1},
(\varphi_\alpha^M)^{-1}
\left(
\bigl(
\rho_\alpha\circ\rho(a_\alpha a^{-1})
\bigr)(u)
\cdot
\varphi_\alpha^M(x)
\right)
\right).
\end{equation}

\begin{lemma}
\label{lem:untwisted-action}
The map~\eqref{eq:untwisted-action} is independent of the choice of the index \(\alpha\) and defines an action of $\mathcal G_M$
on \(\widetilde M\).
\end{lemma}

\begin{proof}
Suppose that \(\widetilde b\in p^{-1}(U_\alpha^B)\cap p^{-1}(U_\beta^B)\). Write \(\widetilde\varphi_\alpha^B(\widetilde b)
=\bigl(p(\widetilde b),a_\alpha\bigr),\) and \(\widetilde\varphi_\beta^B(\widetilde b) = \bigl(p(\widetilde b),a_\beta\bigr)\). Then $a_\alpha=a_{\alpha\beta}a_\beta$
and \(\rho_{\alpha\beta} = \rho_\alpha\circ\rho(a_{\alpha\beta}) \circ\rho_\beta^{-1}\). Since the overlap map $\varphi_{\alpha\beta}^M
=
\varphi_\alpha^M\circ(\varphi_\beta^M)^{-1}$
is \(\rho_{\alpha\beta}\)-equivariant, we have
\[
(\varphi_\alpha^M)^{-1}
\left(
v\cdot\varphi_\alpha^M(x)
\right)
=
(\varphi_\beta^M)^{-1}
\left(
\rho_{\alpha\beta}^{-1}(v)
\cdot\varphi_\beta^M(x)
\right),
\]
for every \(v\in Q^n\). Taking $v=
\bigl(
\rho_\alpha\circ\rho(a_\alpha a^{-1})
\bigr)(u)$,
we obtain
\begin{align*}
\rho_{\alpha\beta}^{-1}
\circ\rho_\alpha\circ\rho(a_\alpha a^{-1})
=
\rho_\beta\circ\rho(a_{\alpha\beta})^{-1}
\circ\rho(a_\alpha a^{-1})\
=
\rho_\beta\circ\rho(a_\beta a^{-1}).
\end{align*}
Therefore
\begin{align*}
&(\varphi_\alpha^M)^{-1}
\left(
\bigl(
\rho_\alpha\circ\rho(a_\alpha a^{-1})
\bigr)(u)
\cdot\varphi_\alpha^M(x)
\right)
=
(\varphi_\beta^M)^{-1}
\left(
\bigl(
\rho_\beta\circ\rho(a_\beta a^{-1})
\bigr)(u)
\cdot\varphi_\beta^M(x)
\right).
\end{align*}
Thus~\eqref{eq:untwisted-action} is independent of the choice of the local chart. It remains to verify the action law. Let $(u_1,a_1),(u_2,a_2)
\in
\mathcal G_M$.
If
\[
\widetilde\varphi_\alpha^B(\widetilde b)
=
\bigl(p(\widetilde b),a_\alpha\bigr),
\]
then
\[
\widetilde\varphi_\alpha^B
(\widetilde b\cdot a_2^{-1})
=\bigl(p(\widetilde b),a_\alpha a_2^{-1}\bigr).
\]
Using the fact that \(\rho(a)\) and \(\rho_\alpha\) are automorphisms of \(Q^n\), the \(Q^n\)-coordinate obtained by applying first \((u_2,a_2)\) and then \((u_1,a_1)\) is
\begin{align*}
&
\bigl(
\rho_\alpha\circ
\rho(a_\alpha a_2^{-1}a_1^{-1})
\bigr)(u_1)
\
\cdot
\bigl(
\rho_\alpha\circ
\rho(a_\alpha a_2^{-1})
\bigr)(u_2)
\
=
\bigl(
\rho_\alpha\circ
\rho(a_\alpha(a_1a_2)^{-1})
\bigr)
\bigl(u_1\rho(a_1)(u_2)\bigr).
\end{align*}
Consequently,
\[
(u_1,a_1)\cdot_\rho
\bigl((u_2,a_2)\cdot_\rho(\widetilde b,x)\bigr) =\bigl((u_1,a_1)(u_2,a_2)\bigr)
\cdot_\rho(\widetilde b,x).
\]
The identity element acts trivially, and hence \eqref{eq:untwisted-action} defines the required action.
\end{proof}

\begin{remark}
The restriction of the action in Lemma~\ref{lem:untwisted-action} to the normal subgroup $Q^n\cong Q^n\times{1}$
is the globally defined locally regular \(Q^n\)-action obtained by untwisting the original local action. Its restriction to
\[
\pi_1(B_M)\cong{1}\times\pi_1(B_M)
\]
is the natural deck-transformation action, twisted by the monodromy representation \(\rho\).
\end{remark}

\begin{definition}\label{def:untwisted-action}
We shall refer to the action of $\mathcal G_M=Q^n\rtimes_\rho\pi_1(B_M)$ on \(\widetilde M\) as the \emph{untwisted action} associated with the local quaternionic torus action \(\mathcal Q\).
\end{definition}

\begin{example}[A local quaternionic torus action with nontrivial monodromy]
\label{ex:local-quaternionic-nonglobal}
Let \(\sigma:Q^2\longrightarrow Q^2\) be the automorphism which interchanges the two factors \(\sigma(q_1,q_2)=(q_2,q_1)\).

Consider the manifold \(\widetilde M=\mathbb R^2\times Q^2.\) We let \(Q^2\) act on \(\widetilde M\) by left multiplication on the \(Q^2\)-factor:
\[
(q_1,q_2)\cdot(t_1,t_2,h_1,h_2)
=
(t_1,t_2,q_1h_1,q_2h_2).
\]
This is the free part of the regular \(Q^2\)-action and has orbit space \(\widetilde B=\mathbb R^2\).

Let \(\mathbb Z^2\) act on \(\widetilde M\) by deck transformations as
follows:
\[
(1,0)\cdot(t_1,t_2,h_1,h_2)
=
(t_1+1,t_2,h_2,h_1),
\]
and
\[
(0,1)\cdot(t_1,t_2,h_1,h_2)
=
(t_1,t_2+1,h_1,h_2).
\]
The quotient
\[
M:=\widetilde M/\mathbb Z^2
\]
is an \(8\)-dimensional manifold. Its orbit space is $B_M\cong \mathbb T^2$,
and the induced local \(Q^2\)-action has monodromy representation
\[
\rho:\pi_1(B_M)\cong\mathbb Z^2
\longrightarrow
\operatorname{Aut}(Q^2)
\]
given by
$\rho(1,0)=\sigma,
\;
\rho(0,1)=\operatorname{id}_{Q^2}$. Since \(\sigma\neq\operatorname{id}_{Q^2}\), the monodromy representation is nontrivial. Hence the corresponding \v{C}ech class
\[
[\rho_{\alpha\beta}]
\in
\check H^1\bigl(B_M;\operatorname{Aut}(Q^2)\bigr)
\]
is nontrivial. By Proposition~\ref{prop:globalization-obstruction}, this local \(Q^2\)-action is not induced by a globally defined locally regular \(Q^2\)-action.

On the other hand, after pulling back to the universal covering
\[
\mathbb R^2\longrightarrow\mathbb T^2,
\]
the monodromy disappears and the local action becomes the globally defined
free \(Q^2\)-action on
\[
\widetilde M=\mathbb R^2\times Q^2.
\]
Thus this example illustrates the untwisting construction of Section~\ref{sec:untwisting}.
\end{example}

\subsection{Lifts to principal quaternionic torus bundles}
\label{sec:local-liftings}

In this subsection, we study the lifting problem for local quaternionic torus actions to principal quaternionic torus bundles. Unless otherwise stated, all spaces, maps, and local actions are assumed to be continuous.

\noindent Keeping our previous notation, let \((M,\mathcal Q)\) be a \(4n\)-dimensional manifold equipped with a local \(Q^n\)-action, and let $p:\widetilde B_M\longrightarrow B_M$
be the universal covering of its orbit space. Fix a representative monodromy representation $\rho:\pi_1(B_M)\longrightarrow \operatorname{Aut}(Q^n)$
of the conjugacy class of representations corresponding to the Čech class $[\{\rho_{\alpha\beta}\}]
\in
\check H^1\bigl(B_M;\operatorname{Aut}(Q^n)\bigr)$.
As in Section~\ref{sec:untwisting}, we put
\[
\widetilde M
=
p^*M
=
\left\{
(\widetilde b,x)\in \widetilde B_M\times M:
p(\widetilde b)=\pi_M(x)
\right\}.
\]
The semidirect product $\mathcal G_M=
Q^n\rtimes_\rho\pi_1(B_M)$
acts on \(\widetilde M\) by the untwisted action constructed in Lemma~\ref{lem:untwisted-action}. We denote the restrictions of this action by
\[
\phi_Q:Q^n\longrightarrow\operatorname{Homeo}(\widetilde M)
\]
and
\[
\phi_\pi:\pi_1(B_M)\longrightarrow\operatorname{Homeo}(\widetilde M).
\]
Note that the two restrictions are related through the following relation; $\forall a\in\pi_1(B_M),\;
u\in Q^n,$
\begin{equation}
\label{eq:base-semidir-compatibility}
\phi_\pi(a)\circ\phi_Q(u)\circ\phi_\pi(a)^{-1}
=
\phi_Q\bigl(\rho(a)(u)\bigr).
\end{equation}

Let $Q^k\longrightarrow P\overset{\pi_P}{\longrightarrow}M$
be a principal quaternionic torus bundle. Its pullback to \(\widetilde M\) is the principal \(Q^k\)-bundle $Q^k\longrightarrow
\widetilde P
\overset{\pi_{\widetilde P}}{\longrightarrow}
\widetilde M$,
where
\[
\widetilde P
=
\left\{
(\widetilde b,p_0)\in\widetilde B_M\times P:
p(\widetilde b)=\pi_M\bigl(\pi_P(p_0)\bigr)
\right\}.
\]

The natural deck-transformation action on \(\widetilde M\) admits a canonical lift
\[
\widetilde\phi_\pi:
\pi_1(B_M)
\longrightarrow
\operatorname{Isom}(\widetilde P)
\]
defined by
\[
\widetilde\phi_\pi(a)(\widetilde b,p_0)
=
(\widetilde b\cdot a^{-1},p_0).
\]
This is a principal \(Q^k\)-bundle automorphism covering
\(\phi_\pi(a)\).

By Section~\ref{sec:untwisting}, the pulled-back local action on \(\widetilde M\) is induced by a globally defined \(Q^n\)-action
\[
\phi_Q:Q^n\longrightarrow\operatorname{Homeo}(\widetilde M).
\]
The deck-transformation action of \(\pi_1(B_M)\) on \(\widetilde M\) has a canonical lift
\[
\widetilde\phi_\pi:\pi_1(B_M)\longrightarrow\operatorname{Isom}(\widetilde P).
\]

\begin{definition}
\label{def:local-Q-lifting}
A \emph{lift to $P$ of the local \(Q^n\)-action \(\mathcal Q\) on M} is a lift of the globalized \(Q^n\)-action to \(\widetilde P\), namely a continuous
homomorphism
\[
\widetilde\phi_Q:
Q^n\longrightarrow
\operatorname{Isom}(\widetilde P),
\]
covering $\phi_Q:Q^n\longrightarrow\operatorname{Homeo}(\widetilde M)$,
such that it is compatible with the lifted deck transformations:
\begin{equation}
\label{eq:lifted-semidir-compatibility}
\widetilde\phi_\pi(a)\circ
\widetilde\phi_Q(u)\circ
\widetilde\phi_\pi(a)^{-1}
=
\widetilde\phi_Q\bigl(\rho(a)(u)\bigr),\;\;\forall a\in\pi_1(B_M),\;u\in Q^n.
\end{equation}
\end{definition}

\noindent Equivalently, condition~\eqref{eq:lifted-semidir-compatibility} says that \(\widetilde\phi_Q\) and \(\widetilde\phi_\pi\) combine to define an action of 
$\mathcal G_M$
on \(\widetilde P\) by principal \(Q^k\)-bundle automorphisms.

\noindent We next give a local description of this notion. Let again $\left\{
(U_\alpha^M,\varphi_\alpha^M)
\right\}_{\alpha\in A}$
be a weakly regular atlas representing \(\mathcal Q\), and let $\left\{
(U_\alpha^B,\varphi_\alpha^B)
\right\}_{\alpha\in A}$
be the induced atlas of \(B_M\). We assume, after refining the covering if necessary, that for every \(\alpha\) there exists a local trivialization
\[
\widetilde\varphi_\alpha^B:
p^{-1}(U_\alpha^B)
\longrightarrow
U_\alpha^B\times\pi_1(B_M),
\]
of the universal covering  $p:\widetilde B_M\longrightarrow B_M.$

\begin{prop}
\label{prop:local-lifting-chartwise}
The principal \(Q^k\)-bundle $Q^k\longrightarrow P\longrightarrow M$ admits a lift of the local \(Q^n\)-action \(\mathcal Q\) if and only if there exists a family
\[
\left\{
(P_\alpha,\widetilde\phi_\alpha,\varphi_\alpha^P)
\right\}_{\alpha\in A}
\]
with the following properties:
\begin{enumerate}
\item[(1)]
\[
Q^k\longrightarrow
P_\alpha
\overset{\pi_\alpha}{\longrightarrow}
\varphi_\alpha^M(U_\alpha^M)
\]
is a locally trivial principal \(Q^k\)-bundle. Thus, for every point \(y\in \varphi_\alpha^M(U_\alpha^M),\) there exists an open neighborhood \(V\subseteq \varphi_\alpha^M(U_\alpha^M)\) and a principal \(Q^k\)-bundle trivialization \(P_\alpha|_V\cong V\times Q^k\).

\item[(2)]
\[
\widetilde\phi_\alpha:
Q^n\longrightarrow
\operatorname{Isom}(P_\alpha)
\]
is a lift of the restriction of the regular \(Q^n\)-action to
\(\varphi_\alpha^M(U_\alpha^M)\);

\item[(3)]
\[
\varphi_\alpha^P:
P|_{U_\alpha^M}
\longrightarrow
P_\alpha
\]
is a principal \(Q^k\)-bundle isomorphism covering
\(\varphi_\alpha^M\);

\item[(4)] On every nonempty overlap \(U_{\alpha\beta}^M\), the induced bundle isomorphism \(\varphi_{\alpha\beta}^P:=\varphi_\alpha^P\circ(\varphi_\beta^P)^{-1}\) between the restricted principal \(Q^k\)-bundles over the corresponding overlap charts is \(\rho_{\alpha\beta}\)-equivariant with respect to the local lifts, that is,
\begin{equation}
\label{eq:local-lift-overlap}
\varphi_{\alpha\beta}^P\circ
\widetilde\phi_\beta(u)
=
\widetilde\phi_\alpha\bigl(\rho_{\alpha\beta}(u)\bigr)
\circ
\varphi_{\alpha\beta}^P,\;\forall u\in Q^n.
\end{equation}
\end{enumerate}
\end{prop}

\begin{proof}
Suppose first that a family satisfying the conditions above, exists. Let $u\in Q^n,\;
(\widetilde b,p_0)\in\widetilde P$
and assume that
\[
\widetilde b\in p^{-1}(U_\alpha^B),
\qquad
\widetilde\varphi_\alpha^B(\widetilde b)
=
\bigl(p(\widetilde b),a_\alpha\bigr).
\]
Define
\begin{equation}
\label{eq:local-lift-from-charts}
\widetilde\phi_Q(u)(\widetilde b,p_0)
:=
\left(
\widetilde b,\,
(\varphi_\alpha^P)^{-1}
\left(
\widetilde\phi_\alpha
\bigl(
(\rho_\alpha\circ\rho(a_\alpha))(u)
\bigr)
\bigl(\varphi_\alpha^P(p_0)\bigr)
\right)
\right).
\end{equation}
We first verify that the definition is independent of the choice of the local chart. Suppose that
\[
\widetilde b \in p^{-1}(U_\alpha^B)\cap p^{-1}(U_\beta^B),
\]
and write
\[
\widetilde\varphi_\alpha^B(\widetilde b) = \bigl(p(\widetilde b),a_\alpha\bigr), \qquad \widetilde\varphi_\beta^B(\widetilde b) = \bigl(p(\widetilde b),a_\beta\bigr).
\]
Then $a_\alpha=a_{\alpha\beta}a_\beta$.
Put $v := (\rho_\beta\circ\rho(a_\beta))(u)$.
By the overlap compatibility condition~\eqref{eq:local-lift-overlap},
\[
\varphi_{\alpha\beta}^P\circ \widetilde\phi_\beta(v) = \widetilde\phi_\alpha\bigl(\rho_{\alpha\beta}(v)\bigr) \circ \varphi_{\alpha\beta}^P.
\]
Moreover,
\begin{align*}
\rho_{\alpha\beta}(v)
&=
\bigl(
\rho_\alpha\circ
\rho(a_{\alpha\beta})\circ
\rho_\beta^{-1}
\bigr)
\bigl(
(\rho_\beta\circ\rho(a_\beta))(u)
\bigr)
\\
&=
\bigl(
\rho_\alpha\circ
\rho(a_{\alpha\beta}a_\beta)
\bigr)(u)
\\
&=
\bigl(
\rho_\alpha\circ\rho(a_\alpha)
\bigr)(u).
\end{align*}
Since $\varphi_{\alpha\beta}^P
=
\varphi_\alpha^P\circ(\varphi_\beta^P)^{-1}$, it follows that
\[(\varphi_\alpha^P)^{-1}
\left(
\widetilde\phi_\alpha
\bigl(
(\rho_\alpha\circ\rho(a_\alpha))(u)
\bigr)
\bigl(\varphi_\alpha^P(p_0)\bigr)
\right)=
(\varphi_\beta^P)^{-1}
\left(
\widetilde\phi_\beta
\bigl(
(\rho_\beta\circ\rho(a_\beta))(u)
\bigr)
\bigl(\varphi_\beta^P(p_0)\bigr)
\right).\]
Hence~\eqref{eq:local-lift-from-charts} is independent of the choice of \(\alpha\) and defines a lift of the global \(Q^n\)-action on \(\widetilde M\).

We next verify compatibility with the deck transformations. Let $a\in\pi_1(B_M),\;
u\in Q^n$,
and suppose that $\widetilde\varphi_\alpha^B(\widetilde b)
=
\bigl(p(\widetilde b),a_\alpha\bigr)$.
With our convention for the right deck action,
\[
\widetilde\varphi_\alpha^B(\widetilde b\cdot a)
=
\bigl(p(\widetilde b),a_\alpha a\bigr)
\]
and hence
\[
\widetilde\varphi_\alpha^B(\widetilde b\cdot a^{-1})
=
\bigl(p(\widetilde b),a_\alpha a^{-1}\bigr).
\]
Applying successively
\(\widetilde\phi_\pi(a)^{-1}\),
\(\widetilde\phi_Q(u)\), and
\(\widetilde\phi_\pi(a)\), we obtain
\begin{align*}
&
\widetilde\phi_\pi(a)\circ
\widetilde\phi_Q(u)\circ
\widetilde\phi_\pi(a)^{-1}
(\widetilde b,p_0)
\\
&=
\left(
\widetilde b,\,
(\varphi_\alpha^P)^{-1}
\left(
\widetilde\phi_\alpha
\bigl(
(\rho_\alpha\circ\rho(a_\alpha a))(u)
\bigr)
\bigl(\varphi_\alpha^P(p_0)\bigr)
\right)
\right)
\\
&=
\left(
\widetilde b,\,
(\varphi_\alpha^P)^{-1}
\left(
\widetilde\phi_\alpha
\bigl(
(\rho_\alpha\circ\rho(a_\alpha))
(\rho(a)(u))
\bigr)
\bigl(\varphi_\alpha^P(p_0)\bigr)
\right)
\right)
\\
&=
\widetilde\phi_Q\bigl(\rho(a)(u)\bigr)
(\widetilde b,p_0).
\end{align*}

Hence \(\widetilde\phi_Q\) is a lift of the local action in the sense of Definition~\ref{def:local-Q-lifting}.

For the opposite direction, suppose that a lift
\[
\widetilde\phi_Q:
Q^n\longrightarrow\operatorname{Isom}(\widetilde P)
\]
satisfying~\eqref{eq:lifted-semidir-compatibility} is given. Put
\[
P_\alpha:=P|_{U_\alpha^M},
\qquad
\pi_\alpha:=\varphi_\alpha^M\circ\pi_P,
\qquad
\varphi_\alpha^P:=\operatorname{id}_{P|_{U_\alpha^M}}.
\]
For \(p_0\in P_\alpha\), and \(\e\in \pi_1(B_M)\) the identity element, 
\[
\left(
(\widetilde\varphi_\alpha^B)^{-1}
\bigl(\pi_M(\pi_P(p_0)),e\bigr),
p_0
\right)
\]
is a point of \(\widetilde P\). We define
\[
\widetilde\phi_\alpha:
Q^n\longrightarrow\operatorname{Isom}(P_\alpha)
\]
by requiring
\begin{align*}
&
\widetilde\phi_Q\bigl(\rho_\alpha^{-1}(u)\bigr)
\left(
(\widetilde\varphi_\alpha^B)^{-1}
\bigl(\pi_M(\pi_P(p_0)),e\bigr),
p_0
\right)
\\
&\qquad =
\left(
(\widetilde\varphi_\alpha^B)^{-1}
\bigl(\pi_M(\pi_P(p_0)),e\bigr),
\widetilde\phi_\alpha(u)(p_0)
\right).
\end{align*}
This definition is well defined because the restriction of the untwisted action to the subgroup $Q^n\cong Q^n\times\{e\}$
fixes the \(\widetilde B_M\)-coordinate. Hence
\(\widetilde\phi_Q(\rho_\alpha^{-1}(u))\) maps a point of the form
\[
\left(
(\widetilde\varphi_\alpha^B)^{-1}
\bigl(\pi_M(\pi_P(p_0)),e\bigr),
p_0
\right)
\]
to another point with the same first coordinate, and therefore uniquely determines \(\widetilde\phi_\alpha(u)(p_0)\). Since \(\widetilde\phi_Q\) is a homomorphism and \(\rho_\alpha^{-1}\) is an automorphism of \(Q^n\), we have
\[
\widetilde\phi_\alpha(uv)
=
\widetilde\phi_\alpha(u)\circ
\widetilde\phi_\alpha(v)
\]
for every \(u,v\in Q^n\). Moreover, each \(\widetilde\phi_\alpha(u)\) is a principal \(Q^k\)-bundle automorphism covering the regular \(Q^n\)-action on
\(\varphi_\alpha^M(U_\alpha^M)\). Thus
\[
\widetilde\phi_\alpha:
Q^n\longrightarrow\operatorname{Isom}(P_\alpha)
\]
is a lift of the local regular action. It remains to verify the compatibility on overlaps. Let
\[
p_0\in P|_{U_{\alpha\beta}^M},
\qquad
b:=\pi_M(\pi_P(p_0))\in B_M.
\]
Using the transition function of the universal covering, we have
\[
(\widetilde\varphi_\beta^B)^{-1}(b,e)
=
(\widetilde\varphi_\alpha^B)^{-1}(b,a_{\alpha\beta}).
\]
Equivalently,
\[
(\widetilde\varphi_\beta^B)^{-1}(b,e)
=
\phi_\pi(a_{\alpha\beta}^{-1})
\left(
(\widetilde\varphi_\alpha^B)^{-1}(b,e)
\right).
\]
Therefore, by the lifted compatibility relation
\eqref{eq:lifted-semidir-compatibility},
\begin{align*}
\left(
(\widetilde\varphi_\beta^B)^{-1}(b,e),
\widetilde\phi_\beta(u)(p_0)
\right)
&=
\widetilde\phi_Q\bigl(\rho_\beta^{-1}(u)\bigr)
\left(
(\widetilde\varphi_\beta^B)^{-1}(b,e),
p_0
\right)
\\
&=
\widetilde\phi_Q\bigl(\rho_\beta^{-1}(u)\bigr)
\widetilde\phi_\pi(a_{\alpha\beta}^{-1})
\left(
(\widetilde\varphi_\alpha^B)^{-1}(b,e),
p_0
\right)
\\
&=
\widetilde\phi_\pi(a_{\alpha\beta}^{-1})
\widetilde\phi_Q
\bigl(
\rho(a_{\alpha\beta})\circ\rho_\beta^{-1}(u)
\bigr)
\left(
(\widetilde\varphi_\alpha^B)^{-1}(b,e),
p_0
\right).
\end{align*}
Since $\rho_{\alpha\beta}
=
\rho_\alpha\circ
\rho(a_{\alpha\beta})\circ
\rho_\beta^{-1}$,
we have $\rho(a_{\alpha\beta})\circ\rho_\beta^{-1}
=
\rho_\alpha^{-1}\circ\rho_{\alpha\beta}$ and hence the last expression can be computed as
\begin{align*}
\left(
(\widetilde\varphi_\beta^B)^{-1}(b,e),
\widetilde\phi_\beta(u)(p_0)
\right)
&=
\widetilde\phi_\pi(a_{\alpha\beta}^{-1})
\widetilde\phi_Q
\bigl(
\rho_\alpha^{-1}(\rho_{\alpha\beta}(u))
\bigr)
\left(
(\widetilde\varphi_\alpha^B)^{-1}(b,e),
p_0
\right)
\\
&=
\left(
(\widetilde\varphi_\beta^B)^{-1}(b,e),
\widetilde\phi_\alpha
\bigl(\rho_{\alpha\beta}(u)\bigr)(p_0)
\right).
\end{align*}
Thus
\[
\widetilde\phi_\beta(u)(p_0)
=
\widetilde\phi_\alpha
\bigl(\rho_{\alpha\beta}(u)\bigr)(p_0).
\]
More generally, with the bundle identifications \(\varphi_\alpha^P\) and \(\varphi_\beta^P\) restored, this becomes
\[
\varphi_{\alpha\beta}^P\circ
\widetilde\phi_\beta(u)
=
\widetilde\phi_\alpha
\bigl(\rho_{\alpha\beta}(u)\bigr)
\circ
\varphi_{\alpha\beta}^P.
\]
Therefore the family
\[
\left\{
(P_\alpha,\widetilde\phi_\alpha,\varphi_\alpha^P) \right\}_{\alpha\in A}
\]
has all the required properties.
\end{proof}

\section{The nonabelian descent obstruction}
\label{sec:nonabelian-descent}

We keep the notation and setup of Section~\ref{sec:local-liftings}. In particular, the pulled-back space \(\widetilde M=p^*M\) carries the globally defined \(Q^n\)-action $\phi_Q:Q^n\longrightarrow\operatorname{Homeo}(\widetilde M)$,
and the pullback bundle $Q^k\longrightarrow\widetilde P\longrightarrow\widetilde M$
carries the canonical lift of the deck-transformation action $\widetilde\phi_\pi:\pi_1(B_M)\longrightarrow
\operatorname{Isom}(\widetilde P)$.
For simplicity, write
\[
D_a:=\widetilde\phi_\pi(a),
\;\; a\in\pi_1(B_M).
\]
With the convention for the right deck action fixed in
Section~\ref{sec:untwisting}, we have $D_{ab}=D_aD_b$. Indeed, for every \((\widetilde b,p_0)\in\widetilde P\),
\begin{align*}
D_aD_b(\widetilde b,p_0)
&=
D_a(\widetilde b\cdot b^{-1},p_0)
\\
&=
(\widetilde b\cdot b^{-1}a^{-1},p_0)
\\
&=
(\widetilde b\cdot (ab)^{-1},p_0)
\\
&=
D_{ab}(\widetilde b,p_0).
\end{align*}

Suppose that the global \(Q^n\)-action on \(\widetilde M\) admits a lift $\widetilde\phi_Q:
Q^n\longrightarrow\operatorname{Isom}(\widetilde P)$ and set $L_u:=\widetilde\phi_Q(u),\;
u\in Q^n$.
The  preliminary lift \(L\) is not assumed to satisfy the compatibility relation with the deck transformations. The purpose of this section is to measure the failure of that compatibility.

Let
\[
\mathcal G(\widetilde P)
:=
\operatorname{Aut}_{\widetilde M}(\widetilde P)
\]
be the gauge group of \(\widetilde P\). As usual, $\mathcal G(\widetilde P)
\cong
\Gamma\bigl(\operatorname{Ad}(\widetilde P)\bigr)$,
where $\operatorname{Ad}(\widetilde P)
=
\widetilde P\times_{Q^k}Q^k$
and \(Q^n\) acts on itself by conjugation. 

\subsection{The descent defect}

To lighten the notation we henceforth set $\Gamma:=\pi_1(B_M)$. For \(a\in\Gamma\) and \(u\in Q^n\), both $D_a^{-1}L_{\rho(a)(u)}D_a$
and $L_u$
cover the same homeomorphism \(\phi_Q(u)\) of \(\widetilde M\). Their difference is therefore a gauge transformation.

\begin{definition}
\label{def:descent-defect}
The \emph{descent defect} associated with the preliminary lift $L:Q^n\longrightarrow\operatorname{Isom}(\widetilde P)$
is the map
\[
\Theta_L:
\Gamma\times Q^n
\longrightarrow
\mathcal G(\widetilde P)
\]
\begin{equation}
\label{eq:descent-defect}
\Theta_L(a,u)
:=
D_a^{-1}L_{\rho(a)(u)}D_aL_u^{-1}.
\end{equation}
\end{definition}

Equivalently,
\begin{equation}
\label{eq:defect-rearranged}
D_a^{-1}L_{\rho(a)(u)}D_a
=
\Theta_L(a,u)L_u.
\end{equation}

The compatibility condition required for a lift of the local action is $D_aL_uD_a^{-1}
=
L_{\rho(a)(u)}$ and thus the preliminary lift \(L\) is compatible with the deck transformations if and only if
\[
\Theta_L(a,u)=1,\;\forall a\in\Gamma,\;u\in Q^n.
\]

Note that the lift \(L\) induces an action of \(Q^n\) on the gauge group by conjugation:
\[
{}^u\gamma
:=
L_u\gamma L_u^{-1},\;\forall u\in Q^n,\gamma\in\mathcal G(\widetilde P).
\]

\begin{lemma}
\label{lem:defect-Q-cocycle}
For every fixed \(a\in\Gamma\), the map
\[
\Theta_L(a,-):
Q^n\longrightarrow\mathcal G(\widetilde P),
\]
satisfies the nonabelian cocycle identity
\begin{equation}
\label{eq:defect-Q-cocycle}
\Theta_L(a,uv)
=
\Theta_L(a,u)\,
{}^u\Theta_L(a,v),\;\;\forall u,v\in Q^n.
\end{equation}
\end{lemma}

\begin{proof}
Since \(L\) is a homomorphism and \(\rho(a)\) is an automorphism of \(Q^n\), we have
\begin{align*}
\Theta_L(a,uv)
&=
D_a^{-1}
L_{\rho(a)(uv)}
D_a
L_{uv}^{-1}
\\
&=
D_a^{-1}
L_{\rho(a)(u)}
L_{\rho(a)(v)}
D_a
L_v^{-1}L_u^{-1}
\\
&=
\left(
D_a^{-1}L_{\rho(a)(u)}D_aL_u^{-1}
\right)
L_u
\left(
D_a^{-1}L_{\rho(a)(v)}D_aL_v^{-1}
\right)
L_u^{-1}
\\
&=
\Theta_L(a,u)\,
L_u\Theta_L(a,v)L_u^{-1}
\\
&=
\Theta_L(a,u)\,
{}^u\Theta_L(a,v).
\end{align*}
\end{proof}

\noindent Thus, for each \(a\in\Gamma\), the defect \(\Theta_L(a,-)\) is a nonabelian \(1\)-cocycle of \(Q^n\) valued in the gauge group. There is also a cocycle-type identity in the \(\Gamma\)-variable.

\begin{lemma}
\label{lem:defect-Gamma-cocycle}
For all \(a,b\in\Gamma\) and \(u\in Q^n\), one has
\begin{equation}
\label{eq:defect-Gamma-cocycle}
\Theta_L(ab,u)
=
D_b^{-1}
\Theta_L\bigl(a,\rho(b)(u)\bigr)
D_b\,
\Theta_L(b,u).
\end{equation}
\end{lemma}

\begin{proof}
Using \(D_{ab}=D_aD_b\) and $\rho(ab)=\rho(a)\circ\rho(b)$,
we obtain
\begin{align*}
\Theta_L(ab,u)
&=
D_{ab}^{-1}
L_{\rho(ab)(u)}
D_{ab}
L_u^{-1}
\\
&=
D_b^{-1}D_a^{-1}
L_{\rho(a)(\rho(b)(u))}
D_aD_b
L_u^{-1}
\\
&=
D_b^{-1}
\left[
D_a^{-1}
L_{\rho(a)(\rho(b)(u))}
D_a
L_{\rho(b)(u)}^{-1}
\right]
D_b
\\
&\qquad\cdot
D_b^{-1}
L_{\rho(b)(u)}
D_b
L_u^{-1}
\\
&=
D_b^{-1}
\Theta_L\bigl(a,\rho(b)(u)\bigr)
D_b\,
\Theta_L(b,u).
\end{align*}
\end{proof}

Equations~\eqref{eq:defect-Q-cocycle} and \eqref{eq:defect-Gamma-cocycle} are the nonabelian counterparts of the two cocycle relations appearing in the theory of lifts of local torus actions \cite{Yoshida1}.

\subsection{Changing the preliminary lift}

Let $\tau:Q^n\longrightarrow\mathcal G(\widetilde P)$ be a continuous map, and define
\[
L_u^\tau:=\tau(u)L_u.
\]
\noindent The map
\[
L^\tau:Q^n\longrightarrow\operatorname{Isom}(\widetilde P),\;\;\;\;u\mapsto L_u^\tau\] is again a lift of the global \(Q^n\)-action  if and only if it is a homomorphism. In fact we have the following.

\begin{lemma}
\label{lem:gauge-modified-global-lift}
 The map
\[
L^\tau:Q^n\longrightarrow\operatorname{Isom}(\widetilde P),
\]
is a homomorphism if and only if
\begin{equation}
\label{eq:tau-crossed-cocycle}
\tau(uv)
=
\tau(u)\,{}^u\tau(v),\;\;\;\forall u,v\in Q^n.
\end{equation}
\end{lemma}

\begin{proof}
We have
\begin{align*}
L_u^\tau L_v^\tau
&=
\tau(u)L_u\tau(v)L_v
\\
&=
\tau(u)
\bigl(L_u\tau(v)L_u^{-1}\bigr)
L_uL_v
\\
&=
\tau(u)\,{}^u\tau(v)L_{uv}.
\end{align*}
Therefore $L_u^\tau L_v^\tau=L_{uv}^\tau$
if and only if $\tau(uv)=\tau(u)\,{}^u\tau(v)$.
\end{proof}

Thus the acceptable modifications of a preliminary lift are the continuous crossed homomorphisms 
\[
\tau:Q^n\longrightarrow\mathcal G(\widetilde P),
\]
with respect to the conjugation action induced by \(L\). We now calculate the descent defect associated with \(L^\tau\).

\begin{lemma}
\label{lem:defect-transformation}
Suppose that \(\tau\) satisfies
\eqref{eq:tau-crossed-cocycle}. Then
\begin{equation}
\label{eq:defect-transformation}
\Theta_{L^\tau}(a,u)
=
D_a^{-1}\tau\bigl(\rho(a)(u)\bigr)D_a\,
\Theta_L(a,u)\,
\tau(u)^{-1},\;\;\forall a\in\Gamma,\;u\in Q^n.
\end{equation}

\end{lemma}

\begin{proof}
By definition of \(L^\tau\), we have
\[
L^\tau_u=\tau(u)L_u,
\qquad
\bigl(L^\tau_u\bigr)^{-1}
=
L_u^{-1}\tau(u)^{-1}.
\]
Therefore
\begin{align*}
\Theta_{L^\tau}(a,u)
&=
D_a^{-1}
L^\tau_{\rho(a)(u)}
D_a
\bigl(L^\tau_u\bigr)^{-1}
\\
&=
D_a^{-1}
\tau\bigl(\rho(a)(u)\bigr)
L_{\rho(a)(u)}
D_a
L_u^{-1}
\tau(u)^{-1}.
\end{align*}
Now insert the identity
$D_aD_a^{-1}=1$
between \(\tau(\rho(a)(u))\) and \(L_{\rho(a)(u)}\). This gives
\begin{align*}
\Theta_{L^\tau}(a,u)
&=
D_a^{-1}
\tau\bigl(\rho(a)(u)\bigr)
D_a
D_a^{-1}
L_{\rho(a)(u)}
D_a
L_u^{-1}
\tau(u)^{-1}
\\
&=
\Bigl(
D_a^{-1}
\tau\bigl(\rho(a)(u)\bigr)
D_a
\Bigr)
\Bigl(
D_a^{-1}
L_{\rho(a)(u)}
D_a
L_u^{-1}
\Bigr)
\tau(u)^{-1}
\\
&=
\Bigl(
D_a^{-1}
\tau\bigl(\rho(a)(u)\bigr)
D_a
\Bigr)
\Theta_L(a,u)
\tau(u)^{-1}.
\end{align*}
\end{proof}

\subsection{Trivialization of the descent defect}

Let $Z_L^1\bigl(Q^n;\mathcal G(\widetilde P)\bigr)$
denote the set of continuous maps $\tau:Q^n\longrightarrow\mathcal G(\widetilde P)$
satisfying
\eqref{eq:tau-crossed-cocycle}
for every \(u,v\in Q^n\). Thus
$Z_L^1\bigl(Q^n;\mathcal G(\widetilde P)\bigr)$
is the set of continuous crossed homomorphisms with respect to the conjugation action induced by \(L\).

\begin{definition}
\label{def:descent-trivializable}
The descent defect \(\Theta_L\) is said to be \emph{trivializable} if there
exists $\tau\in
Z_L^1\bigl(Q^n;\mathcal G(\widetilde P)\bigr)$
such that, for every $a\in\Gamma,u\in Q^n$, the identity
\begin{equation}
\label{eq:descent-trivialization}
\Bigl( D_a^{-1}\, \tau\bigl(\rho(a)(u)\bigr)\, D_a \Bigr) \, \Theta_L(a,u) \, \tau(u)^{-1} =1,
\end{equation}
holds in $\mathcal G(\widetilde P)$.
\end{definition}

Equivalently, \(\Theta_L\) is trivializable if and only if the modified preliminary lift $L_u^\tau:=\tau(u)L_u$
satisfies $D_aL_u^\tau D_a^{-1}
=
L_{\rho(a)(u)}^\tau,\;\;\forall a\in\Gamma,\;u\in Q^n$.

\begin{definition}
\label{def:descent-obstruction}
The \emph{nonabelian descent obstruction} associated with the preliminary lift \(L\) is the obstruction to trivializing the descent defect \(\Theta_L\). We denote it formally by $o_{\mathrm{desc}}(P,L)$. 
We say that
\[
o_{\mathrm{desc}}(P,L)=0
\]
if and only if \(\Theta_L\) is trivializable in the sense of Definition~\ref{def:descent-trivializable}.
\end{definition}

\begin{prop}
\label{prop:descent-independence}
The trivializability of the descent defect is independent of the choice of preliminary lift.
\end{prop}

\begin{proof}
Let
\[
L,L':Q^n\longrightarrow\operatorname{Isom}(\widetilde P)
\]
be two preliminary lifts of the global \(Q^n\)-action. Since \(L_u\) and
\(L'_u\) cover the same homeomorphism of \(\widetilde M\), there exists a unique gauge transformation $\sigma(u)\in\mathcal G(\widetilde P)$
such that $L'_u=\sigma(u)L_u$.
Because both \(L\) and \(L'\) are homomorphisms, \(\sigma\) satisfies $\sigma(uv)
=
\sigma(u)\,{}^u\sigma(v)$.

Suppose first that \(\Theta_L\) is trivializable. Then there exists a
continuous crossed homomorphism
\[
\tau:Q^n\longrightarrow\mathcal G(\widetilde P)
\]
such that the modified lift $L_u^\tau:=\tau(u)L_u$
is compatible with the deck transformations. Define
\[
\tau'(u):=\tau(u)\sigma(u)^{-1}.
\]
Then
\[
\tau'(u)L'_u
=
\tau(u)\sigma(u)^{-1}\sigma(u)L_u
=
\tau(u)L_u
=
L_u^\tau.
\]
Hence the modification of \(L'\) by \(\tau'\) coincides with the compatible lift \(L^\tau\). In particular, \(u\mapsto\tau'(u)L'_u\) is a homomorphism, and therefore \(\tau'\) is a crossed homomorphism with respect to the conjugation action induced by \(L'\), by Lemma~\ref{lem:gauge-modified-global-lift}. Thus \(\Theta_{L'}\) is trivializable. The converse follows by interchanging \(L\) and \(L'\). 
\end{proof}
 In view of Proposition~\ref{prop:descent-independence}, we write
\[
o_{\mathrm{desc}}(P,\mathcal Q)=0
\]
when the descent defect associated with one, and hence every, preliminary lift $L$  of the pulled-back \(Q^n\)-action determined by \(\mathcal Q\) is trivializable.

\subsection{The descent criterion}

We can now state the main result of this section.

\begin{thm}[Nonabelian descent criterion]
\label{thm:nonabelian-descent}
Assume that the global \(Q^n\)-action on \(\widetilde M\) admits a preliminary lift $L$ to the pulled-back principal \(Q^n\)-bundle $Q^n\longrightarrow\widetilde P\longrightarrow\widetilde M$.
Then the local \(Q^n\)-action on \(M\) admits a lift to \(P\) if and only if
\[
o_{\mathrm{desc}}(P,\mathcal Q)=0.
\]
Equivalently, the local action lifts if and only if there exists a continuous
crossed homomorphism
\[
\tau:Q^n\longrightarrow\mathcal G(\widetilde P)
\]
such that
\[
D_a^{-1}\tau\bigl(\rho(a)(u)\bigr)D_a\, \Theta_L(a,u)\, \tau(u)^{-1} = 1,\;\;\forall a\in\Gamma,\;u\in Q^n.
\]

\end{thm}

\begin{proof}
Suppose first that the local \(Q^n\)-action admits a lift to \(P\). By Proposition~\ref{prop:local-lifting-chartwise}, this is equivalent to the existence of a lift
\[
L':Q^n\longrightarrow\operatorname{Isom}(\widetilde P)
\]
of the global \(Q^n\)-action satisfying $D_aL'_uD_a^{-1}
=
L'_{\rho(a)(u)},\;\;\forall a\in\Gamma,u\in Q^n$. Therefore
\[
\Theta_{L'}(a,u)=1.
\]

Let \(L\) be any preliminary lift. There exists a continuous map
\[
\tau:Q^n\longrightarrow\mathcal G(\widetilde P)
\]
such that $L'_u=\tau(u)L_u$.
Since \(L'\) is a homomorphism, \(\tau\) satisfies $\tau(uv)=\tau(u)\,{}^u\tau(v)$.
The transformation formula \eqref{eq:defect-transformation} now gives
\[
1
=
\Theta_{L'}(a,u)
=
D_a^{-1}\tau\bigl(\rho(a)(u)\bigr)D_a\,
\Theta_L(a,u)\,
\tau(u)^{-1}.
\]
Thus $o_{\mathrm{desc}}(P,\mathcal Q)=0$.

Conversely, suppose that $o_{\mathrm{desc}}(P,\mathcal Q)=0$.
Then there exists a continuous crossed homomorphism
\[
\tau:Q^n\longrightarrow\mathcal G(\widetilde P)
\]
such that
\[
D_a^{-1}\tau\bigl(\rho(a)(u)\bigr)D_a\,
\Theta_L(a,u)\,
\tau(u)^{-1}
=
1.
\]
Define
\[
L'_u:=\tau(u)L_u.
\]
By Lemma~\ref{lem:gauge-modified-global-lift}, \(L'\) is a homomorphism and hence a lift of the global \(Q^n\)-action. Moreover, by Lemma~\ref{lem:defect-transformation},
\[
\Theta_{L'}(a,u)=1.
\]
Therefore
\[
D_aL'_uD_a^{-1}
=
L'_{\rho(a)(u)},\;\;\forall a\in\Gamma,\;u\in Q^n.\]
Hence \(L'\) is compatible with the deck transformations and defines a lift of the original local \(Q^n\)-action to \(P\).
\end{proof}

Combining the global lift criterion with Theorem~\ref{thm:nonabelian-descent}, we obtain the complete criterion.

\begin{thm}[Lifting criterion for local quaternionic torus actions]
\label{thm:complete-local-lifting}
Let \(M\) be a manifold equipped with a local \(Q^n\)-action, $Q^k\longrightarrow P\longrightarrow M$ be a principal quaternionic torus bundle, and $Q^k\longrightarrow\widetilde P\longrightarrow\widetilde M$ be its pullback to the untwisted space. Then the local \(Q^n\)-action lifts to \(P\) if and only if the following two conditions hold:
\begin{enumerate}
\item[(1)] the global \(Q^n\)-action on \(\widetilde M\) admits a lift to \(\widetilde P\);

\item[(2)] the associated nonabelian descent obstruction vanishes;
\[
o_{\mathrm{desc}}(P,\mathcal Q)=0.
\]
\end{enumerate}

By Theorem~\ref{thm:global-q-HY}, applied to the case $G=Q^n,\;
X=\widetilde M$,
condition~\textup{(1)} is equivalent to the Borel condition
\[
[\widetilde P]
\in
\mathcal Q^k_{\mathrm{Bor}}\bigl(\widetilde M;Q^n\bigr)
\]
together with the trivializability of the associated global gauge-valued factor set.
\end{thm}

\begin{remark}
For a principal \(T^n\)-bundle, the gauge group is canonically
\[
\mathcal G(\widetilde P)
\cong
C(\widetilde M,T^n),
\]
which is abelian. In this case, Lemma~\ref{lem:defect-Q-cocycle} shows that, for every \(a\in\Gamma\), the map
\[
u\longmapsto\Theta_L(a,u)
\]
is an ordinary \(1\)-cocycle in $Z^1\bigl(T^n;C(\widetilde M,T^n)\bigr)$.
Moreover, Lemma~\ref{lem:defect-Gamma-cocycle} becomes the ordinary \(\Gamma\)-cocycle identity for the induced action of \(\Gamma\) on this abelian group. Hence the descent defect determines a cohomology class
\[
o(P)
\in
H^1\!\left(
\Gamma;
Z^1\bigl(T^n;C(\widetilde M,T^n)\bigr)
\right),
\]
which recovers Yoshida's obstruction class.

For principal \(Q^n\)-bundles, the gauge group
\[
\mathcal G(\widetilde P)
\cong
\Gamma\bigl(\operatorname{Ad}(\widetilde P)\bigr)
\]
is generally nonabelian. Consequently, neither the set of crossed \(Q^n\)-cocycles nor the set of descent defects carries a natural abelian group structure. The obstruction is therefore formulated as the trivializability of a nonabelian descent defect rather than as the vanishing of an ordinary cohomology class.
\end{remark}

 \begin{example}[Descent for an untwisted quotient example]
\label{ex:descent-torus-quotient}
Consider the \(8\)-dimensional manifold
\[
M=\widetilde M/\mathbb Z^2,
\qquad
\widetilde M=\mathbb R^2\times Q^2,
\]
from Example~\ref{ex:local-quaternionic-nonglobal}. Its orbit space is \(B_M\cong \mathbb T^2\), and the monodromy representation \(\rho:\pi_1(B_M)\cong \mathbb Z^2 \longrightarrow \operatorname{Aut}(Q^2)\) is given by \(\rho(1,0)=\sigma\) and \(\rho(0,1)=\operatorname{id}_{Q^2}\), where \(\sigma(q_1,q_2)=(q_2,q_1)\).

Let \(\widetilde P:=\widetilde M\times Q^k\) be the trivial principal \(Q^k\)-bundle over \(\widetilde M\), and let \(\mathbb Z^2\) act on \(\widetilde P\) by lifting its action on \(\widetilde M\) and acting trivially on the structural \(Q^k\)-factor. Explicitly, for the generator \(a=(1,0)\), set
\[
D_a(t_1,t_2,h_1,h_2,g)
=
(t_1-1,t_2,h_2,h_1,g),
\]
and for the generator \(b=(0,1)\), set
\[
D_b(t_1,t_2,h_1,h_2,g)
=
(t_1,t_2-1,h_1,h_2,g).
\]
The quotient \(P:=\widetilde P/\mathbb Z^2\) is then a principal \(Q^k\)-bundle over \(M\), and its pullback to
\(\widetilde M\) is canonically \(\widetilde P\).

The globally defined \(Q^2\)-action on \(\widetilde M\) is
\[
(q_1,q_2)\cdot(t_1,t_2,h_1,h_2)
=
(t_1,t_2,q_1h_1,q_2h_2).
\]
It admits the preliminary lift
\[
L:Q^2\longrightarrow \operatorname{Isom}(\widetilde P)
\]
defined by
\[
L_{(q_1,q_2)}(t_1,t_2,h_1,h_2,g)
=
(t_1,t_2,q_1h_1,q_2h_2,g).
\]
Thus \(L\) acts by the global \(Q^2\)-action on the base and trivially on
the structural \(Q^k\)-factor.

We now compute the descent defect. For \(a=(1,0)\), using
\[
\rho(a)(q_1,q_2)=(q_2,q_1),
\]
we have
\[
D_a^{-1}L_{\rho(a)(q_1,q_2)}D_a
=
L_{(q_1,q_2)}.
\]
Indeed,
\[
D_a(t_1,t_2,h_1,h_2,g)
=
(t_1-1,t_2,h_2,h_1,g),
\]
then
\[
L_{\rho(a)(q_1,q_2)}D_a(t_1,t_2,h_1,h_2,g)
=
(t_1-1,t_2,q_2h_2,q_1h_1,g),
\]
and applying \(D_a^{-1}\) gives
\[
(t_1,t_2,q_1h_1,q_2h_2,g)
=
L_{(q_1,q_2)}(t_1,t_2,h_1,h_2,g).
\]
Hence
\[
D_a^{-1}L_{\rho(a)(u)}D_aL_u^{-1}=1
\]
for every \(u\in Q^2\).

For the second generator \(b=(0,1)\), the monodromy is trivial, \(\rho(b)=\operatorname{id}_{Q^2}\). Since \(D_b\) only translates the \(t_2\)-coordinate and \(L_u\) acts only on the \(Q^2\)-coordinates, we also have
\[
D_b^{-1}L_{\rho(b)(u)}D_b
=
D_b^{-1}L_uD_b
=
L_u.
\]
Therefore
\[
D_b^{-1}L_{\rho(b)(u)}D_bL_u^{-1}=1.
\]

Since \(a\) and \(b\) generate \(\mathbb Z^2\), it follows that
\[
\Theta_L(\gamma,u)
=
D_\gamma^{-1}L_{\rho(\gamma)(u)}D_\gamma L_u^{-1}
=
1
\]
for every \(\gamma\in\pi_1(B_M)\cong\mathbb Z^2\) and every \(u\in Q^2\). Thus the descent defect is identically trivial, \(\Theta_L=1\).

Consequently the descent obstruction vanishes, with trivializing crossed homomorphism \(\tau(u)=1\). By Theorem~\ref{thm:complete-local-lifting}, the local \(Q^2\)-action on \(M\) lifts to the principal \(Q^k\)-bundle \(P\longrightarrow M\).
\end{example}

%%%%%%%%%%%%%%

\end{document}